\documentclass[12pt]{article}
\usepackage{amssymb,amsmath,amsthm, amsfonts}
\usepackage{graphicx}
\usepackage{epsfig}
\usepackage{tikz}

\def\disp{\displaystyle}

\def\dref#1{(\ref{#1})}

\theoremstyle{plain}
\newtheorem{theorem}{Theorem}[section]
\newtheorem{lemma}{Lemma}[section]

\theoremstyle{definition}
\newtheorem{definition}{Definition}[section]

\newtheorem{remark}{Remark}[section]

\numberwithin{equation}{section}

\begin{document}

\title{\bf A new result for global existence and
 boundedness of solutions to a
  parabolic--parabolic Keller--Segel system with logistic source}

\author{
Jiashan Zheng\thanks{Corresponding author.   E-mail address:
 zhengjiashan2008@163.com (J.Zheng)}, YanYan Li
 \\
    School of Mathematics and Statistics Science,\\
     Ludong University, Yantai 264025,  P.R.China \\
}
\date{}

\maketitle \vspace{0.3cm}
\noindent
\begin{abstract}
We consider the following fully parabolic Keller--Segel system with logistic source
$$
 \left\{\begin{array}{ll}
  u_t=\Delta u-\chi\nabla\cdot(u\nabla v)+ au-\mu u^2,\quad
x\in \Omega, t>0,\\
 \disp{v_t=\Delta v- v +u},\quad
x\in \Omega, t>0,\\
 \end{array}\right.\eqno(KS)
$$
over a  bounded domain
 $\Omega\subset\mathbb{R}^N(N\geq1)$,
 with smooth boundary $\partial\Omega$,
 the parameters $a\in \mathbb{R}, \mu>0, \chi>0$.
It is proved that if $\mu>0$, then $(KS)$ admits a global weak solution, while if $\mu>\frac{(N-2)_{+}}{N}\chi C^{\frac{1}{\frac{N}{2}+1}}_{\frac{N}{2}+1}$,
then $(KS)$ possesses a global
classical solution which is bounded,
where $C^{\frac{1}{\frac{N}{2}+1}}_{\frac{N}{2}+1}$ is a positive constant which is corresponding to the maximal Sobolev
regularity. Apart from this, we also show that if $a = 0$ and $\mu>\frac{(N-2)_{+}}{N}\chi C^{\frac{1}{\frac{N}{2}+1}}_{\frac{N}{2}+1}$, then both $u(\cdot, t)$ and $v(\cdot, t)$ decay to zero with respect to
the norm in $L^\infty(\Omega)$ as $t\rightarrow\infty$.
\end{abstract}

\vspace{0.3cm}


\newpage
\section{Introduction}
The Keller--Segel model (see
\cite{Keller2710,Keller79}) has been introduced in order to explain chemotaxis cells aggregation by means of a coupled system of two equations:
 a drift-diffusion type equation for the cells density $u$, and a reaction diffusion equation for the chemoattractant concentration $v$, that is, $(u,v)$ satisfies
 \begin{equation}
 \left\{\begin{array}{ll}
  u_t=\Delta u-\chi\nabla\cdot(u\nabla v),\quad
x\in \Omega, t>0,\\
 \disp{ v_t=\Delta v +u- v},\quad
x\in \Omega, t>0.\\
 \end{array}\right.\label{1.ssdessderrttrrfff1}
\end{equation}
The Keller--Segel models \dref{1.ssdessderrttrrfff1} and their
variants have been extensively studied by
many authors over the past few decades. We refer to the review papers \cite{Bellomo1216,Hillen79,Horstmann2710}
for detailed descriptions of the models and their developments.
The striking feature of Keller--Segel models  is the possibility
of blow-up of solutions in a finite (or infinite) time (see, e.g., \cite{Horstmann2710,Nanjundiahffger79312,Winkler793}), which
strongly depends on the space dimension. A finite (or infinite) time blow-up
never occurs in 1-dimension \cite{Osaki79,Xiang55672gg}  (except in some extreme nonlinear
denerate diffusion model \cite{Cie794}), a critical mass blow-up occurs in 2-dimension:
when the initial mass lies below the threshold solutions exist globally, while
above the threshold solutions blow up in finite time \cite{Horstmann79,Nagai44556710,Senbaffggsssdddssddxxss}, and generic
blow-up in higher-dimensional ($N\geq3$) (\cite{Winkler792,Winkler793}).
For the more related works in this direction, we mention that a corresponding quasilinear
version or the signal is  consumed
by the cells
has been deeply investigated by  Cie\'{s}lak et al.  \cite{Cie794,Cie791,Cie1sseerr7102},  Winkler et al. \cite{Bellomo1216,Tao794,Winkler79,Winkler72} and Zheng et al.  \cite{Zhengzseeddd0,Zhengssssssdefr23}.

In order to investigate the growth of the population,
considerable effort has been devoted to Keller--Segel models with the logistic term. For example, Winkler (\cite{Winkler37103})
 proposed
and
investigated the following fully parabolic Keller--Segel
system with logistic source 
\begin{equation}
 \left\{\begin{array}{ll}
  u_t=\Delta u-\chi\nabla\cdot(u\nabla v)+f(u),\quad
x\in \Omega, t>0,\\
 \disp{\tau v_t=\Delta v +u- v},\quad
x\in \Omega, t>0,\\
 \disp{\frac{\partial u}{\partial \nu}=\frac{\partial v}{\partial \nu}=0},\quad
x\in \partial\Omega, t>0,\\
\disp{u(x,0)=u_0(x)},\quad \tau v(x,0)=\tau v_0(x),~~
x\in \Omega\\
 \end{array}\right.\label{1.ssderrfff1}
\end{equation}
with
 $\tau= 1$,
  where,
$\Omega\subset\mathbb{R}^N(N\geq1)$ is a bounded  domain 
 with smooth boundary
$\partial\Omega$ and
$\frac{\partial}{\partial\nu}$ denoted the derivative with respect to the outward normal vector $\nu$ of $\partial\Omega$.
The
kinetic term $f$ describes
cell proliferation and death (simply referred to as growth).
Hence, many efforts have been made first
for the linear chemical production and the logistic source:
\begin{equation}
f(u) = au -\mu  u^2.\label{sffsddggyuusderrsddffftyuu1.1}
\end{equation}

During the past  decade, the Keller--Segel models of type \dref{1.ssderrfff1}
have been studied extensively
by many authors, where the main issue of the investigation is whether the
solutions of the models are bounded or blow-up
 (see e.g., Cie\'{s}lak et al.
 \cite{Cie72,Cie794,Cie791,Cie793}, Burger et al. \cite{Burger2710},
  Calvez and  Carrillo \cite{Calvez710},  Keller and Segel \cite{Keller2710,Keller79},
   Horstmann et al. \cite{Horstmann2710,Horstmann79,Horstmann791},
 Osaki \cite{Osaki79,Osakix391},
    Painter and Hillen \cite{Painter79}, Perthame \cite{Perthame710},
Rascle and Ziti \cite{Rascle710}, Wang et al. \cite{Wang79,Wang76}, Winkler \cite{Winkler555ani710,Winkler79,Winkler37103,Winkler715,Winkler793,Winkler79312vv}, Zheng \cite{Zheng}).
If $\tau = 0$, \dref{1.ssderrfff1} is referred to as simplified parabolic--elliptic
chemotaxis system which is physically relevant when the chemicals diffuse much
faster than cells do.
Tello and  Winkler (\cite{Tello710})  mainly proved that
 that the weak solutions of \dref{1.ssderrfff1} ($\tau = 0$ in \dref{1.ssderrfff1}) exist for arbitrary $\mu > 0$
and that they are smooth and globally bounded if the logistic damping effect satisfies  ${\bf \mu > \frac{(N-2)_+}{N}\chi}$.
However, it is shown by some recent
studies that the nonlinear diffusion
(see Mu et al. \cite{Wang76,Zhenge455678887}) and the (generalized) logistic damping (see Winkler \cite{Winkler715},  Li and  Xiang \cite{LiLittffggsssdddssddxxss}, Zheng \cite{Zheng0}) may prevent the blow-up of solutions.

Turning to the parabolic-parabolic system ($\tau=1$ in \dref{1.ssderrfff1}), for any $\mu> 0,$ it is known, at least, that
all solutions of \dref{1.ssderrfff1} are bounded when $N = 1$ (see Osaki and  Yagi \cite{Osaki79}) or
$N=2$ (see Osaki et al. \cite{Osakix391}).
In light of deriving a
bound for the quantity
$$\Sigma_{k=0}^mb_{k}\int_{\Omega}u^{k}|\nabla v|^{2m-2k}
 $$
with arbitrarily large $m\in \mathbb{N}$ and appropriately constructed positive $b_0, \ldots, b_m$, Winkler (\cite{Winkler37103}) proved that \dref{1.ssderrfff1} admits a unique, smooth and bounded solution if
 $\mu$ is {\bf large enough} and $N\geq1$. However, he did not give the lower bound estimation for the logistic source.
 If  $\Omega\subset \mathbb{R}^N(N\geq1)$ is a smooth bounded {\bf convex} domain, Lankeit (\cite{Lankeitberg79}) proved that \dref{1.ssderrfff1} ($f(u) = au -\mu  u^2$ in \dref{1.ssderrfff1}) admits a global weak solutions for any $\mu > 0$, while if  $a$ is appropriately small and $N = 3,$ the global weak solutions
which eventually become smooth and decay in both components (\cite{Lankeitberg79}).
To the best
of our knowledge, it is yet unclear whether for $\Omega$ is a {\bf non-convex} domain, $N\geq3$   and small values of $\mu > 0$ certain initial data may
enforce finite-time blow-up of solutions.


%

 In this paper, we prove that \dref{1.ssderrfff1}  admits a unique, smooth and bounded solution if the logistic source  $\mu>$ ${\bf\frac{(N-2)_{+}}{N}\chi C^{\frac{1}{\frac{N}{2}+1}}_{\frac{N}{2}+1}}$,
where $C^{\frac{1}{\frac{N}{2}+1}}_{\frac{N}{2}+1}$ is a positive constant which is corresponding to the maximal Sobolev
regularity. This result implies that  the global boundedness of the solution for the  complete parabolic--parabolic and parabolic--elliptic models, which
need a coefficient of the logistic source to keep the same (except a constant $C_{\frac{N}{2}+1}^{\frac{1}{\frac{N}{2}+1}}$).
  Some recent
studies show that  nonlinear diffusion (Xiang \cite{Xiangssdd55672gg}, Viglialoro and Woolley \cite{Viglia2345loroqqqeesserr3344455678887},
Wang et al.  \cite{Wang72}, Winkler \cite{Winklerqqqeesserr3344455678887},
  Zheng  \cite{Zheng33312186,Zhengsssddsseedssddxxss}), or also (generalized) logistic dampening (Lankeit \cite{Lankeit3344455678887}, Nakaguchi and  Osaki  \cite{Nakagussdchi3710},
Viglialoro et al. \cite{Viglialorosserr3344455678887,Viglialoroqqqeesserr3344455678887,Viglialoro3344455678887}, Zheng and Wang \cite{Zhengsssddssddsseedssddxxss}) may
prevent blow-up of solutions.

Going beyond the basic knowledge of above boundedness results, some important
findings were given by many authors which assert that the interaction effects
between cross-diffusion and cell kinetics may result in quite a colorful dynamics
(see e.g. Winkler  et al. \cite{Tao44621,Winkler79312vv,Winkler79312}, Galakhov et al. \cite{Galakhov1230}, Zheng \cite{Zhengsssddswwerrrseedssddxxss}).
For example,
Osaki etal. (\cite{Osakix391,Osaki79,Osaki17102}) studied
 the boundedness and large time behavior of solutions of the model \dref{1.ssderrfff1} on dimension $N\leq2$.
For the parabolic--elliptic case ($\tau=0$ in \dref{1.ssderrfff1}), in \cite{Tello710}, Tello and Winkler proved that  the equilibrium $(1, 1)$ is a global attractor if
$\mu >2\chi$ and $a =\mu$.
While for the parabolic--parabolic case ($\tau=1$ in \dref{1.ssderrfff1}), assume  
the ratio $\frac{\mu}{\chi}$
 is
 sufficiently
 large, Winkler (\cite{Winkler79312}) proved that
 the
 unique
 nontrivial
 spatially
 homogeneous
 equilibrium
 given
 by
 $u=v\equiv \frac{1}{\mu}$ is globally asymptotically stable in the sense that for any choice of suitably regular nonnegative initial data $(u_0,v_0)$
 such that $u_0\not\equiv0$.

Inspired by these researches, the purpose of this paper is to show the global solvability of classical (or weak) solutions to the following problem:
 \begin{equation}
 \left\{\begin{array}{ll}
  u_t=\Delta u-\chi\nabla\cdot(u\nabla v)+au-\mu u^2,\quad
x\in \Omega, t>0,\\
 \disp{ v_t=\Delta v +u- v},\quad
x\in \Omega, t>0,\\
 \disp{\frac{\partial u}{\partial \nu}=\frac{\partial v}{\partial \nu}=0},\quad
x\in \partial\Omega, t>0,\\
\disp{u(x,0)=u_0(x)},\quad  v(x,0)= v_0(x),~~
x\in \Omega.\\
 \end{array}\right.\label{1.1}
\end{equation}
 The main novel lies in the $L^\infty$ estimate of $u$, we use careful analysis,  the variation-of-constants formula and a variation of Maximal Sobolev Regularity
  to  develop some $L^p$-estimate techniques to raise the a priori estimate of solutions from $L^{p_0}(\Omega)$$(p_0>\frac{N}{2})$$\rightarrow$ $L^{p}(\Omega)$(for all $p>1)$ (see Lemmata \ref{lemma45630223116}--\ref{lemma45630223}), and then  combining with the Moser iteration method (see e.g.  Lemma A.1 of \cite{Tao794}), we finally established the $L^\infty$ bound of $u$ (see the proof of  Theorem \ref{theorem3}).

\section{Preliminaries and  main results}

%
%
In order to prove the main results, we first state several elementary
lemmas which will be needed later.
\begin{lemma}(\cite{Hajaiej,Ishida})\label{lemma41ffgg}
Let  $s\geq1$ and $q\geq1$.
Assume that $p >0$ and $a\in(0,1)$ satisfy
$$\frac{1}{2}-\frac{p}{N}=(1-a)\frac{q}{s}+a(\frac{1}{2}-\frac{1}{N})~~\mbox{and}~~p\leq a.$$
Then there exist $c_0, c'_0 >0$ such that for all $u\in W^{1,2}(\Omega)\cap L^{\frac{s}{q}}(\Omega)$,
$$\|u\|_{W^{p,2}(\Omega)} \leq c_{0}\|\nabla u\|_{L^{2}(\Omega)}^{a}\|u\|^{1-a}_{L^{\frac{s}{q}}(\Omega)}+c'_0\|u\|_{L^{\frac{s}{q}}(\Omega)}.$$
\end{lemma}

\begin{lemma}\label{lemma45xy1222232}(\cite{Cao,Hieber})
Suppose  $\gamma\in (1,+\infty)$ and $g\in L^\gamma((0, T); L^\gamma(
\Omega))$. 
On the other hand,
assuming $v$ is a solution of the following initial boundary value
 \begin{equation}
 \left\{\begin{array}{ll}
v_t -\Delta v+v=g,\\
\disp\frac{\partial v}{\partial \nu}=0,\\
v(x,0)=v_0(x).\\
 \end{array}\right.\label{1.3xcx29}
\end{equation}
Then there exists a positive constant $C_\gamma$ such that if $s_0\in[0,T)$, $v(\cdot,s_0)\in W^{2,\gamma}(\Omega)$ with $\disp\frac{\partial v(\cdot,s_0)}{\partial \nu}=0,$ then
\begin{equation}
\begin{array}{rl}
&\disp{\int_{s_0}^Te^{\gamma s}\|\Delta v(\cdot,t)\|^{\gamma}_{L^{\gamma}(\Omega)}ds}\\
\leq &\disp{C_\gamma\left(\int_{s_0}^Te^{\gamma s}
\|g(\cdot,s)\|^{\gamma}_{L^{\gamma}(\Omega)}ds+
e^{\gamma s_0}(\|v_0(\cdot,s_0)\|^{\gamma}_{L^{\gamma}(\Omega)}+\|\Delta v_0(\cdot,s_0)\|^{\gamma}_{L^{\gamma}(\Omega)})\right).}\\
\end{array}
\label{cz2.5bbv114}
\end{equation}
\end{lemma}

Our first result concerns the global weak existence of solutions and reads as follows.
\begin{theorem}\label{theorem773} Let $\Omega\subset\mathbb{R}^N(N\geq1)$ be a smooth bounded  domain.
Assume
that $u_0\in C^0(\bar{\Omega})$ and $v_0\in W^{1,\theta}(\bar{\Omega})$ (with some $\theta> n$) both are nonnegative.
  If 
$\mu>0$,
%
%
%
%
then it holds that
there exists at least one global weak
solution (in the sense of Definition \ref{df1} below) of problem \dref{1.1}.
\end{theorem}

\begin{remark} We  remove the convexity of $\Omega$
 required in \cite{Lankeitberg79}.
\end{remark}

Moreover, if in addition we assume that $\mu>\frac{(N-2)_+}{N}\chi C_{\frac{N}{2}+1}^{\frac{1}{\frac{N}{2}+1}}$, then our solutions will actually be bounded and smooth
and hence classical:
\begin{theorem}\label{theorem3} Let $\Omega\subset\mathbb{R}^N(N\geq1)$ be a smooth bounded  domain. Assume
that $u_0\in C^0(\bar{\Omega})$ and $v_0\in W^{1,\theta}(\bar{\Omega})$ (with some $\theta> n$) both are nonnegative.
  If 
$\mu>\frac{(N-2)_+}{N}\chi C_{\frac{N}{2}+1}^{\frac{1}{\frac{N}{2}+1}}$,
%
%
%
%
then \dref{1.1}
possesses a unique classical solution $(u,v)$ which is globally bounded in  $\Omega\times(0,\infty)$.
\end{theorem}
\begin{remark}


%


 (i) Theorem \ref{theorem3}  extends the results of   Winkler (\cite{Winkler37103}), who proved the possibility of boundness, in the cases
 $\mu>0$ is sufficiently large, and with $\Omega\subset \mathbb{R}^N$ is a
convex bounded domains.

 (ii) Theorem \ref{theorem3} asserts that, as in the corresponding two-dimensional Keller-Segel system
(see Osaki et al. \cite{Osakix391}), even arbitrarily small quadratic degradation of cells (for any $\mu>0$) is sufficient to rule out blow-up and rather
ensure boundedness of solutions.

(iii) 
%
From Theorem \ref{theorem3}, we derive that for the complete parabolic--parabolic and parabolic--elliptic models, the global boundedness of the solutions need the coefficient of the logistic source  keep the same (which differs from a constant $C_{\frac{N}{2}+1}^{\frac{1}{\frac{N}{2}+1}}$).

\end{remark}

\begin{theorem}\label{theoremfgb3}  Let $\Omega\subset\mathbb{R}^N(N\geq1)$ be a smooth bounded  domain.
Let
 $a= 0$, and suppose that $\mu>\frac{(N-2)_+}{N}\chi C_{\frac{N}{2}+1}^{\frac{1}{\frac{N}{2}+1}}$.
  Then as long as $u_0\in C^0(\bar{\Omega})$ and $v_0\in W^{1,\theta}(\bar{\Omega})$ (with some $\theta> n$) both are nonnegative, the global bounded solution $(u,v)$
 constructed in Theorem \ref{theorem3} satisfies
\begin{equation}
\|u(\cdot,t)\|_{L^\infty(\Omega)}\rightarrow0,~~\|v(\cdot,t)\|_{L^\infty(\Omega)}\rightarrow0
\label{ggbbdffff1.1fghyuisdakkklll}
\end{equation}
as $t\rightarrow\infty$.
\end{theorem}
\begin{remark}We find that if (the coefficient
 of  logistic source) $\mu>\frac{(N-2)_+}{N}\chi C_{\frac{N}{2}+1}^{\frac{1}{\frac{N}{2}+1}},$  then Theorem \ref{theoremfgb3} holds
for any $N\geq1$, hence in this paper, we drop the hypothesis of dimension $N=3$ which is required by Theorems 1.3--1.4 of  \cite{Lankeitberg79}.


\end{remark}

In order to discuss  the global weak
solution  for any  
{\bf $\mu>0$} (see the proof of Lemma \ref{kkklemmaghjmk4563025xxhjklojjkkk}),
  we need to consider an appropriately  approximated system
    of \dref{1.1} at first.
 Indeed, 
the corresponding approximated problem is introduced as follows:
\begin{equation}
 \left\{\begin{array}{ll}
  u_{\varepsilon t}=\Delta u_{\varepsilon}-\chi\nabla\cdot(u_{\varepsilon}F_{\varepsilon}(u_{\varepsilon})\nabla v_{\varepsilon})+ u_{\varepsilon}(a-\mu u_{\varepsilon}),\quad
x\in \Omega, t>0,\\
 \disp{v_{\varepsilon t}=\Delta v_{\varepsilon} +u_{\varepsilon}- v_{\varepsilon}},\quad
x\in \Omega, t>0,\\
 \disp{\frac{\partial u_{\varepsilon}}{\partial \nu}=\frac{\partial v_{\varepsilon}}{\partial \nu}=0},\quad
x\in \partial\Omega, t>0,\\
\disp{u_{\varepsilon}(x,0)=u_0(x)},v_{\varepsilon}(x,0)=v_0(x),\quad
x\in \Omega,\\
 \end{array}\right.\label{1.1fghyuisda}
\end{equation}
where
\begin{equation}
F_{\varepsilon}(s)=\frac{1}{1+\varepsilon s}~~\mbox{for all}~~s \geq 0~~\mbox{and}~~\varepsilon> 0.
\label{1.ffggvbbnxxccvvn1}
\end{equation}
%

The following local existence result is rather standard, since a similar reasoning in
 \cite{Cao,Cie72,Wang79,Wang76,Wang72,Zheng0}.
 We omit it here.
\begin{lemma}\label{lemma70}Let $\Omega\subset\mathbb{R}^N(N\geq1)$ be a smooth bounded  domain.
Assume that the nonnegative functions $u_0\in C^{0}(\bar{\Omega})$ and $v_0\in W^{1,\theta}(\bar{\Omega})$ (with some $\theta> N$).
Then there exist a
maximal $T_{max,\varepsilon}\in(0,\infty]$ and a uniquely determined pair $(u_{\varepsilon},v_{\varepsilon})$ of nonnegative functions
$$
 \left\{\begin{array}{ll}
 u_{\varepsilon}\in C^0(\bar{\Omega}\times[0,T_{max,\varepsilon}))\cap C^{2,1}(\bar{\Omega}\times(0,T_{max,\varepsilon})),\\
  v_{\varepsilon}\in  C^0(\bar{\Omega}\times[0,T_{max,\varepsilon}))\cap C^{2,1}(\bar{\Omega}\times(0,T_{max,\varepsilon}))\cap
   L^\infty_{loc}((0,T_{max,\varepsilon}); W^{1,\theta}(\Omega))\\
   \end{array}\right.
   $$
that solve \dref{1.1fghyuisda} in the classical sense in
 $\Omega\times(0, T_{max,\varepsilon})$. 
%
%
%
Moreover, if  $T_{max,\varepsilon}<+\infty$, then
\begin{equation}
\|u_{\varepsilon}(\cdot, t)\|_{L^\infty(\Omega)}\rightarrow\infty~~ \mbox{as}~~ t\nearrow T_{max,\varepsilon}
\label{1.163072x}
\end{equation}
is fulfilled.
\end{lemma}


\begin{definition} \label{df1} 
We call $(u, v)$ a global weak solution of \dref{1.1} if
\begin{equation}
 \left\{\begin{array}{ll}
   u\in L_{loc}^1(\bar{\Omega}\times[0,\infty)),\\
    v \in  L_{loc}^1([0,\infty);W^{1,1}(\Omega)),\\
 \end{array}\right.\label{dffff1.1fghyuisdakkklll}
\end{equation}
such that  $u\geq 0$ and $v\geq 0$ a.e. in
$\Omega\times(0, \infty)$,
\begin{equation}\label{726291hh}
\begin{array}{rl}
&\nabla u~~ \mbox{and}~~~~~u\nabla v~~ \mbox{belong to}~~
L^1_{loc}(\bar{\Omega}\times [0, \infty)),\\
\end{array}
\end{equation}
 and that
\begin{equation}
\begin{array}{rl}\label{eqx45xx12112ccgghh}
\disp{-\int_0^{T}\int_{\Omega}u\varphi_t-\int_{\Omega}u_0\varphi(\cdot,0)  }=&\disp{-
\int_0^T\int_{\Omega}\nabla u\cdot\nabla\varphi+\chi\int_0^T\int_{\Omega}u
\nabla v\cdot\nabla\varphi}\\
&\disp{+\int_0^T\int_{\Omega} u(a-\mu u)\varphi}\\
\end{array}
\end{equation}
for any $\varphi\in C_0^{\infty} (\bar{\Omega}\times[0, \infty))$
  as well as
  \begin{equation}
\begin{array}{rl}\label{eqx45xx12112ccgghhjj}
\disp{-\int_0^{T}\int_{\Omega}v\varphi_t-\int_{\Omega}v_0\varphi(\cdot,0)  }=&\disp{-
\int_0^T\int_{\Omega}\nabla v\cdot\nabla\varphi-\int_0^T\int_{\Omega}(v-u)\varphi.}\\
\end{array}
\end{equation}
\end{definition}
\section{The global weak
solution of \dref{1.1}}

In this section, we are going to establish an iteration step to develop the main ingredient of our result.
The iteration depends on a series of a-priori estimates.
To this end, we first show the following Lemma,
which is presented below for the sake of completeness and easy reference (see also  Lemma 2.1 of \cite{Winkler37103}).
\begin{lemma}\label{wsdelemma45}
Under the assumptions in 
 Lemma  \ref{lemma70}, we derive that
there exists a positive constant 
$C$ independent of $\varepsilon$
such that the solution of \dref{1.1fghyuisda} satisfies
%
%
\begin{equation}
\int_{\Omega}{u_{\varepsilon}(x,t)}+\int_{\Omega} {v^2_{\varepsilon}}(x,t)+\int_{\Omega}|\nabla {v_{\varepsilon}}(x,t)|^2 \leq C~~\mbox{for all}~~ t\in(0, T_{max,\varepsilon})
\label{cz2.5ghju48cfg924ghyuji}
\end{equation}
and
\begin{equation}
\int_t^{t+\tau}\int_{\Omega}[|\nabla {v_{\varepsilon}}|^2+u^2_{\varepsilon}+ |\Delta {v_{\varepsilon}}|^2]\leq C~~\mbox{for all}~~ t\in(0, T_{max,\varepsilon}-\tau)
\label{223455cz2.5ghjddffgggu48cfg924ghyuji}
\end{equation}
with
\begin{equation}
\tau:=\min\{1,\frac{1}{6}T_{max,\varepsilon}\}.
\label{223455cz2.5seerttyyuughjddffgggu48cfg924ghyuji}
\end{equation}
Moreover,
for each $T\in(0, T_{max,\varepsilon})$, one can find a constant $C > 0$ independent of $\varepsilon$ such that
\begin{equation}
\begin{array}{rl}
&\disp{\int_{0}^T\int_{\Omega}[|\nabla {v_{\varepsilon}}|^2+u^2_{\varepsilon}+ |\Delta {v_{\varepsilon}}|^2]\leq C.}\\
\end{array}
\label{bnmbncz2.5ghhjuyuivvbssdddeennihjj}
\end{equation}
\end{lemma}
\begin{proof}
Here and throughout the proof of Lemma \ref{wsdelemma45}, we shall denote by $C_i (i\in \mathbb{N})$ several positive
constants independent of  $\varepsilon$.
From integration of the first equation in \dref{1.1fghyuisda} we obtain
\begin{equation}
\begin{array}{rl}
\disp\frac{d}{dt}\disp\int_{\Omega}u_{\varepsilon}=&\disp{\int_\Omega   (au_{\varepsilon}-\mu u_{\varepsilon}^2) ~~\mbox{for all}~~ t\in(0, T_{max,\varepsilon}),}\\
\end{array}
\label{111811cz2.5114114}
\end{equation}
which combined with  the Cauchy-Schwarz inequality implies that
\begin{equation}
\begin{array}{rl}
\disp\frac{d}{dt}\disp\int_{\Omega}u_{\varepsilon}\leq&\disp{a\int_\Omega u_{\varepsilon}-\frac{\mu}{|\Omega|} \left(\int_\Omega u_{\varepsilon}\right)^2
 ~~\mbox{for all}~~ t\in(0, T_{max,\varepsilon}).}\\
\end{array}
\label{1ssdd11811cz2.5114114}
\end{equation}
Hence, employing  the Young inequality to  \dref{1ssdd11811cz2.5114114} and  integrating the resulted inequality   in time, we derive  that there exists a positive constant $C_1$ independent of $\varepsilon$ such that
\begin{equation}
\int_{\Omega}{u_{\varepsilon}(x,t)} \leq C_1~~\mbox{for all}~~ t\in(0, T_{max,\varepsilon}).
\label{s3344cz2.5ghju48cfg924ghyuji}
\end{equation}
For  each $T\in(0, T_{max,\varepsilon})$,
we integrate \dref{111811cz2.5114114} over $(0,T)$
 and  recall \dref{s3344cz2.5ghju48cfg924ghyuji} to obtain
\begin{equation}
\begin{array}{rl}
&\disp{\int_{0}^T\int_{\Omega}u^2_{\varepsilon}\leq C_2.}\\
\end{array}
\label{bnmbncz2.5ghhdderrffjuyuivvbssdddeennihjj}
\end{equation}
Moreover, integrating  \dref{111811cz2.5114114} over $(t,t+\tau)$
 and  using  \dref{s3344cz2.5ghju48cfg924ghyuji}, we also derive 
 \begin{equation}
\begin{array}{rl}
&\disp{\int_{t}^{t+\tau}\int_{\Omega}u^2_{\varepsilon}\leq C_3~~\mbox{for all}~~ t\in(0, T_{max,\varepsilon}-\tau),}\\
\end{array}
\label{234bnmbncz2.5ghhdderrffjuyuivvbssdddeennihjj}
\end{equation}
where $\tau$ is given by \dref{223455cz2.5seerttyyuughjddffgggu48cfg924ghyuji}. 
Now,
multiplying the second  equation of \dref{1.1fghyuisda}
  by $-\Delta v_{\varepsilon}$, integrating over $\Omega$ and using  the Young inequality, 
 we get
$$
\begin{array}{rl}
\disp\frac{1}{{2}}\frac{d}{dt}\|\nabla v_{\varepsilon}\|^{{2}}_{L^{{2}}(\Omega)}+\int_{\Omega}|\Delta v_{\varepsilon}|^2+\int_{\Omega}|\nabla v_{\varepsilon}|^2
=&\disp{-\int_\Omega  u_{\varepsilon}\Delta v_{\varepsilon} }\\
\leq&\disp{\frac{1}{2}\int_\Omega  u_{\varepsilon}^2+\frac{1}{2}\int_\Omega  |\Delta v_{\varepsilon}|^2 ~~\mbox{for all}~~ t\in(0, T_{max,\varepsilon}),}\\
\end{array}
$$
which in light of \dref{234bnmbncz2.5ghhdderrffjuyuivvbssdddeennihjj} and  Lemma 2.3 of \cite{Tao41215} implies that
\begin{equation}
\int_{\Omega}|\nabla {v_{\varepsilon}}(x,t)|^2 \leq C_3~~\mbox{for all}~~ t\in(0, T_{max,\varepsilon})
\label{cz2.5ghju48cfg92dffff4ghdffffhhyuji}
\end{equation}
and
\begin{equation}
\begin{array}{rl}
&\disp{\int_{0}^T\int_{\Omega}[|\nabla {v_{\varepsilon}}|^2+ |\Delta {v_{\varepsilon}}|^2]\leq C_4.}\\
\end{array}
\label{bnmddfgghghbncz2.5ghhjufffffyuivvbssdddeennihjj}
\end{equation}
Next, testing the second  equation of \dref{1.1fghyuisda}
  by $v_{\varepsilon}$ and applying \dref{234bnmbncz2.5ghhdderrffjuyuivvbssdddeennihjj}, we conclude that
\begin{equation}
\int_{\Omega}v_{\varepsilon}^2(x,t) \leq C_5~~\mbox{for all}~~ t\in(0, T_{max,\varepsilon}).
\label{cz2.5ghju48cfg924ghdffffggggffhhyuji}
\end{equation}
Now, collecting \dref{s3344cz2.5ghju48cfg924ghyuji}--\dref{cz2.5ghju48cfg924ghdffffggggffhhyuji} yields to \dref{cz2.5ghju48cfg924ghyuji}
and \dref{bnmbncz2.5ghhjuyuivvbssdddeennihjj}. Finally, the same argument as in the derivation of \dref{bnmbncz2.5ghhjuyuivvbssdddeennihjj} then shows that \dref{223455cz2.5ghjddffgggu48cfg924ghyuji} holds. 
\end{proof}

\begin{lemma}\label{lemmaghjssddgghhmk4563025xxhjklojjkkk}
Under the conditions of  Lemma  \ref{lemma70},
there exists $C>0$  independent of $\varepsilon$ such that the solution of \dref{1.1fghyuisda} satisfies
\begin{equation}
\begin{array}{rl}
&\disp{\int_{\Omega}u_{\varepsilon}\ln u_{\varepsilon}\leq C}\\
\end{array}
\label{czfvgb2.5ghhjuyuiissswwerrtthjj}
\end{equation}
for all $t\in(0, T_{max,\varepsilon})$.
Moreover,  for each $T\in(0, T_{max,\varepsilon})$,
 one can find a constant $C > 0$ independent of $\varepsilon$
  such that
\begin{equation}
\begin{array}{rl}
&\disp{\int_{0}^T\int_{\Omega}  \frac{|\nabla u_{\varepsilon}|^2}{u_{\varepsilon}}\leq C(T+1)}\\
\end{array}
\label{bnmbncz2.5ghhjuyuiihjj}
\end{equation}
%
as well as
\begin{equation}
\begin{array}{rl}
&\disp{\int_{0}^T\int_{\Omega}   u^2_{\varepsilon}(\ln u_{\varepsilon}+1)\leq C(T+1).}\\
\end{array}
\label{cvffvbssdgvvcz2.ffghhjj5ghhjuyuiihjj}
\end{equation}

\end{lemma}
\begin{proof}
First, testing the first equation in \dref{1.1fghyuisda} by $\ln u_{\varepsilon}$ yields
\begin{equation}
\begin{array}{rl}
&\disp{\frac{d}{dt}\int_{\Omega}u_{\varepsilon}\ln u_{\varepsilon}}\\
=&\disp{\int_{\Omega}u_{\varepsilon t}\ln u_{\varepsilon}+u_{\varepsilon t}}\\
=&\disp{-\int_{\Omega}\frac{|\nabla u_{\varepsilon}|^2}{u_{\varepsilon}}+\chi\int_{\Omega}  F_{\varepsilon}(u_{\varepsilon})\nabla u_{\varepsilon}\cdot\nabla v_{\varepsilon}-\mu\int_{\Omega}u^2_{\varepsilon}\ln u_{\varepsilon}}\\
&\disp{+a\int_{\Omega}u_{\varepsilon}\ln u_{\varepsilon}-\mu\int_{\Omega}u^2_{\varepsilon}+a\int_{\Omega}u_{\varepsilon}~~~\mbox{for all}~~ t\in (0, T_{max,\varepsilon}).}\\
\end{array}
\label{czfvgb2.5ghhjuyuddffccvviihjj}
\end{equation}
Next, letting 
the function   $\psi:[0,\infty)\rightarrow \mathbb{R}$ be defined by
$$
 \psi(s)=\left\{\begin{array}{ll}
   -\frac{\mu}{2} s^2-\mu s^2\ln s+\frac{\mu}{2}s^2\ln (s+1)+as\ln s,~~~s>0,\\
    0,~~~s=0.\\
 \end{array}\right.
 $$
Then
$$\disp\lim_{s\rightarrow+\infty}\frac{\psi(s)}{s^2\ln (s+1)}=-\frac{\mu}{2},$$
so that for some $s_0 > 0$ we have $\psi < 0$ on $(s_0,\infty).$ Since clearly $\psi$ is continuous on $[0,\infty),$
hence, we derive that
\begin{equation}
 -\frac{\mu}{2} s^2-\mu s^2\ln s+as\ln s\leq \frac{\mu}{2}s^2\ln (s+1)+C_1 ~~~\mbox{for all}~~ s > 0.
\label{czfvgb2.5ghhjuyddfggghhhudssdddssddeedffccvviihjj}
\end{equation}
On the other hand, employing \dref{czfvgb2.5ghhjuyddfggghhhudssdddssddeedffccvviihjj} and using the Young inequality and \dref{cz2.5ghju48cfg924ghyuji}, one can get 
\begin{equation}
\begin{array}{rl}
&\disp{-\mu\int_{\Omega}u_{\varepsilon}^2-\mu\int_{\Omega}u_{\varepsilon}^2\ln u_{\varepsilon}+a\int_{\Omega}u_{\varepsilon}\ln u_{\varepsilon}+a\int_{\Omega}u_{\varepsilon}}\\
\leq&\disp{-\frac{\mu}{2}\int_{\Omega}u_{\varepsilon}^2\ln (u_{\varepsilon}+1)-\frac{\mu}{2}\int_{\Omega}u_{\varepsilon}^2+C_2
~~~\mbox{for all}~~ t\in (0, T_{max,\varepsilon})}\\
\end{array}
\label{czfvgb2.5ghhjuyudssddeedffccvviihjj}
\end{equation}
with some positive constant  $C_2.$
Next, 
once more integrating by parts and using the Young inequality and \dref{1.ffggvbbnxxccvvn1}, we derive
\begin{equation}
\begin{array}{rl}
\chi\disp\int_{\Omega}  F_{\varepsilon}(u_{\varepsilon})\nabla u_{\varepsilon}\cdot\nabla v_{\varepsilon}
=&\disp{-\chi\int_{\Omega} \int_0^{u_{\varepsilon}}   \frac{1}{1+\varepsilon s}ds \Delta v_{\varepsilon}}\\
\leq&\disp{\chi\int_{\Omega} \int_0^{u_{\varepsilon}}   \frac{1}{1+\varepsilon s}ds |\Delta v_{\varepsilon}|}\\
\leq&\disp{\chi\int_{\Omega} u_{\varepsilon}|\Delta v_{\varepsilon}|}\\
\leq&\disp{\frac{\mu}{4}\int_{\Omega}   u_{\varepsilon}^2+\frac{\chi^2}{\mu}\int_{\Omega}   |\Delta v_{\varepsilon}|^2~~~\mbox{for all}~~ t\in (0, T_{max,\varepsilon}).}\\
\end{array}
\label{czfvgb2.5ghhjssdeeaauyudssddeedffccvviihjj}
\end{equation}

Putting the estimates \dref{czfvgb2.5ghhjuyudssddeedffccvviihjj} and  \dref{czfvgb2.5ghhjssdeeaauyudssddeedffccvviihjj} into \dref{czfvgb2.5ghhjuyuddffccvviihjj} and using \dref{cz2.5ghju48cfg924ghyuji}, then there exists a positive constant $C_{3}$ such that
$$
\begin{array}{rl}
&\disp{\frac{d}{dt}\int_{\Omega}u_{\varepsilon}\ln u_{\varepsilon}+\int_{\Omega}\frac{|\nabla u_{\varepsilon}|^2}{u_{\varepsilon}}+\frac{\mu}{4}\int_{\Omega}u^2_{\varepsilon}\ln (u_{\varepsilon}+1)+
\frac{\mu}{4}\int_{\Omega}u^2_{\varepsilon}}\\
\leq&\disp{\frac{\chi^2}{\mu}\int_{\Omega} |\Delta v_{\varepsilon}|^2+C_{3}~~~\mbox{for all}~~ t\in (0, T_{max,\varepsilon}),}\\
\end{array}
$$
which implies that
\begin{equation}
\begin{array}{rl}
&\disp{\frac{d}{dt}\int_{\Omega}u_{\varepsilon}\ln u_{\varepsilon}+\int_{\Omega}u_{\varepsilon}\ln u_{\varepsilon}+\int_{\Omega}\frac{|\nabla u_{\varepsilon}|^2}{u_{\varepsilon}}+\frac{\mu}{8}\int_{\Omega}u^2_{\varepsilon}\ln (u_{\varepsilon}+1)+
\frac{\mu}{4}\int_{\Omega}u^2_{\varepsilon}}\\
\leq&\disp{\frac{\chi^2}{\mu}\int_{\Omega} |\Delta v_{\varepsilon}|^2+C_{4}~~~\mbox{for all}~~ t\in (0, T_{max,\varepsilon})}\\
\end{array}
\label{czfvgb2.5ghhjussderfrfyuddffccvviihjj}
\end{equation}
and some positive constant $C_4$. Here we have use the fact that   $$\disp\lim_{s\rightarrow+\infty}\frac{s \ln s}{s^2\ln (s+1)}=0.$$
Combined with \dref{223455cz2.5ghjddffgggu48cfg924ghyuji} and \dref{czfvgb2.5ghhjussderfrfyuddffccvviihjj}, applying   Lemma 2.3 of \cite{Tao41215} (see also Lemma 2.4 of \cite{Zhengsssddssddsseedssddxxss}), we can obtain
\dref{czfvgb2.5ghhjuyuiissswwerrtthjj}--\dref{cvffvbssdgvvcz2.ffghhjj5ghhjuyuiihjj}. The proof of Lemma \ref{wsdelemma45} is completed.
\end{proof}

With Lemma  \ref{lemmaghjssddgghhmk4563025xxhjklojjkkk} at hand, using the idea coming from \cite{Zhengssedddrrrrzseeddd0}, we are now in the position to prove  the solution of approximate
problem \dref{1.1fghyuisda}  is actually global in time.
\begin{lemma}\label{kkklemmaghjmk4563025xxhjklojjkkk} Let $\Omega\subset\mathbb{R}^N(N\geq1)$ be a smooth bounded  domain.
If 
$\mu>0$, then
for all $\varepsilon\in(0,1),$ the solution of  \dref{1.1fghyuisda} is global in time.
\end{lemma}
\begin{proof}
In this Lemma, we shall denote by $C_i (i\in \mathbb{N})$ various positive constants which may vary from step to
step and which possibly depend on $\varepsilon$.
Assuming that $T_{max,\varepsilon}<+\infty.$ 
Then,
we first note that as a particular
consequence of Lemmata \ref{wsdelemma45}--\ref{lemmaghjssddgghhmk4563025xxhjklojjkkk}, we can then find $C_1 > 0$ such that
\begin{equation}
\begin{array}{rl}
&\disp{\int_{0}^{T_{max,\varepsilon}}\int_{\Omega}[|\nabla {v_{\varepsilon}}|^2+u^2_{\varepsilon}+ |\Delta {v_{\varepsilon}}|^2]\leq C_1.}\\
\end{array}
\label{23455bnmbncz2.5ghhjuyuivvbssdddeenddrftgyuuiinihjj}
\end{equation}
Multiplying the first equation of \dref{1.1fghyuisda}
  by $u_{\varepsilon}^{{p}-1}$ and integrating over $\Omega$, 
 we get
\begin{equation}
\begin{array}{rl}
&\disp{\frac{1}{{p}}\frac{d}{dt}\|u_{\varepsilon}\|^{{p}}_{L^{{p}}(\Omega)}+({{p}-1})\int_{\Omega}u_{\varepsilon}^{{{p}-2}}|\nabla u_{\varepsilon}|^2}
\\
=&\disp{-\chi\int_\Omega \nabla\cdot(u_{\varepsilon}F_{\varepsilon}(u_{\varepsilon})\nabla v_{\varepsilon})
  u_{\varepsilon}^{{p}-1} +
\int_\Omega   u_{\varepsilon}^{{p}-1}(au_{\varepsilon}-\mu u_{\varepsilon}^2) ~~\mbox{for all}~~ t\in(0, T_{max,\varepsilon}).}\\
\end{array}
\label{11cz2.5114114}
\end{equation}

Next,
integrating by parts to the first term on the right hand side of \dref{11cz2.5114114},  using  the Young inequality and \dref{1.ffggvbbnxxccvvn1},
we obtain 
\begin{equation}
\begin{array}{rl}
&\disp{-\chi\int_\Omega \nabla\cdot(u_{\varepsilon}F_{\varepsilon}(u_{\varepsilon})\nabla v_{\varepsilon})
  u_{\varepsilon}^{{p}-1} }
\\
=&\disp{({{p}-1})\chi\int_\Omega  u_{\varepsilon}^{{{p}-1}}F_{\varepsilon}(u_{\varepsilon})\nabla u_{\varepsilon}\cdot\nabla v_{\varepsilon} }
\\
=&\disp{({{p}-1})\chi\int_\Omega  \nabla\int_0^{u_{\varepsilon}}\frac{\tau^{{{p}-1}}}{1+\varepsilon \tau}d\tau\cdot\nabla v_{\varepsilon} }
\\
=&\disp{-({{p}-1})\chi\int_\Omega  \int_0^{u_{\varepsilon}}\frac{\tau^{{{p}-1}}}{1+\varepsilon \tau}d\tau\Delta v_{\varepsilon} }
\\
\leq&\disp{({{p}-1})\frac{\chi}{\varepsilon}\int_\Omega  \int_0^{u_{\varepsilon}}\tau^{p-2}d\tau|\Delta v_{\varepsilon}| }
\\
\leq&\disp{\frac{\chi}{\varepsilon}\int_\Omega  u_{\varepsilon}^{p-1}|\Delta v_{\varepsilon}| }
\\
\leq&\disp{\frac{\mu}{2}\int_\Omega u_{\varepsilon}^{{p+1}}dx+C_2\int_\Omega|\Delta v_{\varepsilon}|^{\frac{p+1}{2}}.}
\\
\end{array}
\label{11cz2.563019114}
\end{equation}
Inserting \dref{11cz2.563019114} into \dref{11cz2.5114114} and using the Young inequality, we derive
\begin{equation}
\begin{array}{rl}
&\disp{\frac{1}{{p}}\frac{d}{dt}\|u_{\varepsilon}\|^{{p}}_{L^{{p}}(\Omega)}+({{p}-1})\int_{\Omega}u_{\varepsilon}^{{{p}-2}}|\nabla u_{\varepsilon}|^2\leq-\frac{\mu}{4}\int_\Omega u_{\varepsilon}^{{p+1}}+C_2\int_\Omega|\Delta v_{\varepsilon}|^{\frac{p+1}{2}} +
C_3.}\\
\end{array}
\label{11cz2.51ssertyyyu14114ssde}
\end{equation}
Next, choosing $p=2$ in \dref{11cz2.51ssertyyyu14114ssde} and employing \dref{23455bnmbncz2.5ghhjuyuivvbssdddeenddrftgyuuiinihjj}, we conclude that
\begin{equation}
\| u_{\varepsilon}(\cdot,t)\|_{L^2(\Omega)}\leq C_4~~\mbox{for all}~~ t\in(0, T_{max,\varepsilon}).
\label{ddxcvbbggdddfghhdfgcz2vv.5ghju48cfg924ghyusdffji}
\end{equation}
Employing the same arguments as in the proof of Lemma 4.1 in \cite{Horstmann791}, and taking advantage of \dref{ddxcvbbggdddfghhdfgcz2vv.5ghju48cfg924ghyusdffji} and Lemma \ref{lemma70}, we conclude the estimate
\begin{equation}
\|\nabla v_{\varepsilon}(\cdot,t)\|_{L^{\gamma_0}(\Omega)}\leq C_5~~\mbox{for all}~~ t\in(0, T_{max,\varepsilon})~~\mbox{and}~~\gamma_0<\frac{2N}{(N-2)_{+}}.
\label{ddxcvbbggdddfghhdfgcz2vv.5ghju48cfg9ddfgghjj24ghyusdffji}
\end{equation}
Next,
integrating by parts to the first term on the right hand side of \dref{11cz2.5114114},  using \dref{1.ffggvbbnxxccvvn1} and the Young inequality,
we obtain 
\begin{equation}
\begin{array}{rl}
&\disp{-\chi\int_\Omega \nabla\cdot(u_{\varepsilon}F_{\varepsilon}(u_{\varepsilon})\nabla v_{\varepsilon})
  u_{\varepsilon}^{{p}-1} }
\\
=&\disp{({{p}-1})\chi\int_\Omega  u_{\varepsilon}^{{{p}-1}}F_{\varepsilon}(u_{\varepsilon})\nabla u_{\varepsilon}\cdot\nabla v_{\varepsilon} }
\\
\leq&\disp{({{p}-1})\frac{\chi}{\varepsilon}\int_\Omega  u_{\varepsilon}^{{{p}-2}}|\nabla u_{\varepsilon}||\nabla v_{\varepsilon}| }
\\
\leq&\disp{\frac{({{p}-1})}{4}\int_\Omega  u_{\varepsilon}^{{{p}-2}}|\nabla u_{\varepsilon}|^2+C_6\int_\Omega  u_{\varepsilon}^{{{p}-2}}|\nabla v_{\varepsilon}|^2,}
\\
\end{array}
\label{11111cz2.563019114}
\end{equation}
which together with  \dref{11cz2.5114114}, the Young inequality  and the H\"{o}lder  inequality implies that
\begin{equation}
\begin{array}{rl}
&\disp{\frac{1}{{p}}\frac{d}{dt}\|u_{\varepsilon}\|^{{p}}_{L^{{p}}(\Omega)}+\frac{3({{p}-1})}{4}\int_{\Omega}u_{\varepsilon}^{{{p}-2}}|\nabla u_{\varepsilon}|^2+\frac{\mu}{2}\int_\Omega   u_{\varepsilon}^{{p}+1} }\\
\leq&\disp{C_6\int_\Omega  u_{\varepsilon}^{{{p}-2}}|\nabla v_{\varepsilon}|^2 +
C_7}\\
\leq&\disp{C_6\left(\int_\Omega  u_{\varepsilon}^{\frac{\gamma_0({p}-2)}{\gamma_0-2}}\right)^{\frac{\gamma_0-2}{\gamma_0}}\left(\int_\Omega|\nabla v_{\varepsilon}|^{\gamma_0} \right)^{\frac{2}{\gamma_0}}+
C_7}\\
\leq&\disp{C_8\left(\int_\Omega  u_{\varepsilon}^{\frac{\gamma_0({p}-2)}{\gamma_0-2}}\right)^{\frac{\gamma_0-2}{\gamma_0}}+
C_7~~\mbox{for all}~~ t\in(0, T_{max,\varepsilon}).}\\
\end{array}
\label{11111cz2.51ssertyyyu14114ssde}
\end{equation}
By  the Gagliardo--Nirenberg inequality, we derive 
\begin{equation}
\begin{array}{rl}
&\disp C_8\left(\int_\Omega  u_{\varepsilon}^{\frac{\gamma_0({p}-2)}{\gamma_0-2}}\right)^{\frac{\gamma_0-2}{\gamma_0}}\\\
=&\disp{ C_8\|   u_{\varepsilon}^{\frac{p}{2}}\|^{\frac{2(p-2)}{p}}_{L^{\frac{2\gamma_0(p-2)}{p(\gamma_0-2)}}(\Omega)}}\\
\leq&\disp{C_{9}(\|\nabla    u_{\varepsilon}^{\frac{p}{2}}\|_{L^2(\Omega)}^{\tilde{\mu}_1}\|  u_{\varepsilon}^{\frac{p}{2}}\|_{L^\frac{2}{p}(\Omega)}^{1-\tilde{\mu}_1}+\|  u_{\varepsilon}^{\frac{p}{2}}\|_{L^\frac{2}{p}(\Omega)})^{\frac{2(p-2)}{p}}}\\
\leq&\disp{C_{10}(\|\nabla    u_{\varepsilon}^{\frac{p}{2}}\|_{L^2(\Omega)}^{\frac{2(p-2)\tilde{\mu}_1}{p}}+1)}\\
=&\disp{C_{10}(\|\nabla    u_{\varepsilon}^{\frac{p}{2}}\|_{L^2(\Omega)}^{2\frac{N\gamma_0(p-3)-2N}{\gamma_0(Np-N+2)}}+1)~~\mbox{for all}~~ t\in(0, T_{max,\varepsilon}),}\\
\end{array}
\label{cz2.wwsdeddfvgbhnjerfvghyh}
\end{equation}
where
$$\tilde{\mu}_1=\frac{\frac{Np}{2}-\frac{Np(\gamma_0-2)}{2\gamma_0(p-2)}}{1-\frac{N}{2}+\frac{Np}{2}}\in(0,1)~~\mbox{and}~~
\frac{2(p-2)}{p}\frac{\frac{Np}{2}-\frac{Np(\gamma_0-2)}{2\gamma_0(p-2)}}{1-\frac{N}{2}+\frac{Np}{2}}=2\frac{N(p-3)-\frac{2N}{\gamma_0}}{Np-N+2}<2.$$
In view  of \dref{cz2.wwsdeddfvgbhnjerfvghyh} and the Young inequality, we derive that
\begin{equation}
\begin{array}{rl}
\disp C_8\left(\int_\Omega  u_{\varepsilon}^{\frac{\gamma_0({p}-2)}{\gamma_0-2}}\right)^{\frac{\gamma_0-2}{\gamma_0}}\leq \frac{({{p}-1})}{4}\int_{\Omega}u_{\varepsilon}^{{{p}-2}}|\nabla u_{\varepsilon}|^2+C_{11}~\mbox{for all}~t\in(0, T_{max,\varepsilon}),\\
\end{array}
\label{cz2.wwddfvgbsdeddfvgbhnjerfvghyh}
\end{equation}
which
 together with \dref{11111cz2.51ssertyyyu14114ssde} yields that
\begin{equation}
\begin{array}{rl}
&\disp{\frac{1}{{p}}\frac{d}{dt}\|u_{\varepsilon}\|^{{{p}}}_{L^{{p}}(\Omega)}+
\frac{({{p}-1})}{2}\int_{\Omega}u_{\varepsilon}^{{{p}-2}}|\nabla u_{\varepsilon}|^2+\frac{\mu}{2}\int_\Omega   u_{\varepsilon}^{{p}+1} \leq
C_{12}~\mbox{for all}~t\in(0, T_{max,\varepsilon}). }\\
\end{array}
\label{cz2.5ghhjuyuiwssxxddcc22ssxccnnihjj}
\end{equation}
Now, with some basic analysis, we may derive that for all $p>1,$ 
\begin{equation}
\|u_{\varepsilon}(\cdot,t)\|_{L^p(\Omega)}\leq C_{13}~~\mbox{for all}~~ t\in(0, T_{max,\varepsilon}).
\label{ddxcvbbggdddfghhdfddcfvgbbgcz2vv.5gttghju48cfg924ghyuji}
\end{equation}

Next, using the outcome of \dref{ddxcvbbggdddfghhdfddcfvgbbgcz2vv.5gttghju48cfg924ghyuji} with suitably large $p$ as a starting point, we may employ
a Moser-type iteration (see  e.g. Lemma A.1 of  \cite{Tao794}) applied to the first equation of \dref{1.1fghyuisda} to derive
\begin{equation}
\begin{array}{rl}
\|u_{\varepsilon}(\cdot, t)\|_{L^{{\infty}}(\Omega)}\leq C_{14} ~~ \mbox{for all}~~~  t\in(\rho,T_{max,\varepsilon}) \\
\end{array}
\label{222cz2.5g5gghh56789hhjui78jj90099}
\end{equation}
with any  positive constant $\rho$.
In view of \dref{222cz2.5g5gghh56789hhjui78jj90099}, we apply Lemma \ref{lemma70} to reach a contradiction.
\end{proof}


In this subsection, we provide some time-derivatives uniform estimates of solutions to the
system \dref{1.1}. The estimate is
used in this Section  to construct the  weak solution of the equation \dref{1.1}.
This will be the purpose of the following  lemma:

\begin{lemma}\label{qqqqlemma45630hhuujjuuyytt}
 Then for any $T>0, $
  one can find $C > 0$ independent if $\varepsilon$ such that 
  \begin{equation}
\begin{array}{rl}
\disp\int_{0}^T\disp\int_{\Omega}|\nabla u_{\varepsilon}|^{\frac{4}{3}}
\leq&\disp{C(T+1),}\\
\end{array}
\label{5555ddffbnmbncz2ddfvgffgssssdddddddtyybhh.htt678ghhjjjddfghhhyuiihjj}
\end{equation}
\begin{equation}
 \begin{array}{ll}
\disp\int_0^T\|\partial_tu_{\varepsilon}(\cdot,t)\|_{(W^{2,q}(\Omega))^*}dt  \leq C(T+1)\\
   \end{array}\label{1.1ddfgeddvbnmklllhyuisda}
\end{equation}
as well as
\begin{equation}
 \begin{array}{ll}
  \disp\int_0^T\|\partial_tv_{\varepsilon}(\cdot,t)\|_{(W^{1,2}(\Omega))^*}^{2}dt  \leq C(T+1)\\
   \end{array}\label{wwwwwqqqq1.1dddfgbhnjmdfgeddvbnmklllhyussddisda}
\end{equation}
and
\begin{equation}
 \begin{array}{ll}
  \disp\int_0^T\int_{\Omega}|u_{\varepsilon}F_{\varepsilon}(u_{\varepsilon})\nabla v_{\varepsilon}| \leq C(T+1).\\
   \end{array}\label{1.1dddfgbhnjmdfgeddvbnmklllhyussddisda}
\end{equation}
\end{lemma}
\begin{proof}
Firstly,  due to \dref{cz2.5ghju48cfg924ghyuji}, \dref{bnmbncz2.5ghhjuyuivvbssdddeennihjj}, \dref{bnmbncz2.5ghhjuyuiihjj}, employing  the
H\"{o}lder inequality  and the Gagliardo-Nirenberg inequality, we conclude  that there exist positive
 constants $C_{1}$ and $C_2$ 
such that
\begin{equation}
\begin{array}{rl}
\disp\int_{0}^T\disp\int_{\Omega}|\nabla u_{\varepsilon}|^{\frac{4}{3}}=&\disp{
\int_{0}^T\int_{\Omega}\frac{|\nabla u_{\varepsilon}|^{\frac{4}{3}}}{u_{\varepsilon}^{\frac{2}{3}}}u_{\varepsilon}^{\frac{2}{3}}}\\
\leq&\disp{C_1\left[\int_{0}^T\int_{\Omega}\frac{|\nabla u_{\varepsilon}|^{2}}{u_{\varepsilon}} \right]^{\frac{2}{3}}
\left[\int_{0}^T\int_{\Omega}u_{\varepsilon}^{2} \right]^{\frac{1}{3}} }\\
\leq&\disp{C_{2}(T+1)~~\mbox{for all}~~ T > 0.}\\
\end{array}
\label{5555ddffbnmbncz2ddfvgffgssdddtyybhh.htt678ghhjjjddfghhhyuiihjj}
\end{equation}
Next,
testing the first equation of \dref{1.1}
 by certain   $\varphi\in C^{\infty}(\bar{\Omega})$ and using \dref{1.ffggvbbnxxccvvn1}, we have
 \begin{equation}
\begin{array}{rl}
&\disp\left|\int_{\Omega}u_{\varepsilon t}\varphi\right|\\
 =&\disp{\left|\int_{\Omega}\left[\Delta u_{\varepsilon}-\chi\nabla\cdot(u_{\varepsilon}F_{\varepsilon}(u_{\varepsilon})\nabla v_{\varepsilon})+ u_{\varepsilon}(a- \mu u_{\varepsilon})\right]\varphi\right|}
\\
\leq&\disp{\int_\Omega|\nabla u_{\varepsilon}||\nabla\varphi|+\chi\int_\Omega u_{\varepsilon}|\nabla v_{\varepsilon}||\nabla \varphi|
+\int_\Omega [au_{\varepsilon}+\mu u^2_{\varepsilon}]|\varphi|}\\
\leq&\disp{\left\{\int_\Omega \left[|\nabla u_{\varepsilon}|+ \chi u_{\varepsilon}|\nabla v_{\varepsilon}|+au_{\varepsilon}+\mu u^2_{\varepsilon}\right]\right\}\|\varphi\|_{W^{1,\infty}(\Omega)}}\\
\end{array}
\label{gbhncvbmdcfvgcz2.5ghju48}
\end{equation}
for all $t>0$.
Hence, observe that the embedding $W^{2,q }(\Omega)\hookrightarrow  W^{1,\infty}(\Omega)(q>N)$, due to \dref{cz2.5ghju48cfg924ghyuji}, \dref{bnmbncz2.5ghhjuyuivvbssdddeennihjj}, \dref{5555ddffbnmbncz2ddfvgffgssdddtyybhh.htt678ghhjjjddfghhhyuiihjj}, applying  the Young inequality, we deduce $C_3$ and $C_4$ such
that
\begin{equation}
\begin{array}{rl}
&\disp\int_0^T\|\partial_{t}u_{\varepsilon}(\cdot,t)\|_{(W^{2,q }(\Omega))^*}dt\\
\leq&\disp{C_3\left\{\int_0^T\int_{\Omega}|\nabla u_{\varepsilon}|^{\frac{4}{3}}+\int_0^T\int_{\Omega}u^2_{\varepsilon}+\int_0^T\int_{\Omega}|\nabla v_{\varepsilon}|^{2}\right\}
}\\
\leq&\disp{C_4(T+1)~~\mbox{for all}~~ T > 0,
}\\
\end{array}
\label{yyygbhncvbmdcxxcdfvgbfvgcz2.5ghju48}
\end{equation}
which implies  \dref{1.1ddfgeddvbnmklllhyuisda}.

Likewise, given any $\varphi\in  C^\infty(\bar{\Omega})$, we may test the second equation in \dref{1.1} against
$\varphi$
to conclude  that
\begin{equation}
\begin{array}{rl}
\disp\left|\int_{\Omega}\partial_{t}v_{\varepsilon}(\cdot,t)\varphi\right|=&\disp{\left|\int_{\Omega}\left[\Delta v_{\varepsilon}-v_{\varepsilon}+u_{\varepsilon}\right]\cdot\varphi\right|}
\\
=&\disp{\left|-\int_\Omega \nabla v_{\varepsilon}\cdot\nabla\varphi-\int_\Omega  v_{\varepsilon}\varphi+\int_\Omega u_{\varepsilon} \varphi\right|}\\
\leq&\disp{\left\{\|\nabla v_{\varepsilon}\|_{L^{{2}}(\Omega)}+\|v_{\varepsilon} \|_{L^{2}(\Omega)}+\|u_{\varepsilon}\|_{L^{2}(\Omega)}\right\}\|\varphi\|_{W^{1,2}(\Omega)}
~~\mbox{for all}~~ t>0.}\\
\end{array}
\label{wwwwwqqqqgbhncvbmdcfxxdcvbccvbbvgcz2.5ghju48}
\end{equation}
Collecting \dref{cz2.5ghju48cfg924ghyuji} and \dref{bnmbncz2.5ghhjuyuivvbssdddeennihjj}, we infer from \dref{wwwwwqqqqgbhncvbmdcfxxdcvbccvbbvgcz2.5ghju48}
\begin{equation}
\begin{array}{rl}
&\disp\int_0^T\|\partial_{t}v_{\varepsilon}(\cdot,t)\|^{2}_{(W^{1,2}(\Omega))^*}dt\\
\leq&\disp{C_{5}\left(\int_0^T\int_\Omega|\nabla v_{\varepsilon}|^{2}+\int_0^T\int_\Omega u^2_{\varepsilon}+
\int_0^T\int_\Omega v_{\varepsilon}^{2}\right)}\\
\leq&\disp{C_{6}(T+1)
~~\mbox{for all}~~ T>0}\\
\end{array}
\label{wwwwwqqqqgbhncvbmdcfxxxcvbddfghnxdcvbbbvgcz2.5ghju48}
\end{equation}
and  some positive constants $C_5$, $C_6.$
Therefore,
we see \dref{wwwwwqqqq1.1dddfgbhnjmdfgeddvbnmklllhyussddisda}  holds immediately.

In light of
 \dref{cz2.5ghju48cfg924ghyuji}, \dref{bnmbncz2.5ghhjuyuivvbssdddeennihjj}
   and   the Young inequality, we derive that  there exists a positive
constant $C_{7}$ such that 
\begin{equation}
\begin{array}{rl}
\disp\int_{0}^T\disp\int_{\Omega}|u_{\varepsilon} F_{\varepsilon}(u_{\varepsilon})\nabla v_{\varepsilon}|
\leq&\disp{\left(\int_0^T\int_{\Omega}|\nabla v_{\varepsilon}|^{2}\right)^{\frac{1}{2}}\left(\int_0^T\int_{\Omega}u_{\varepsilon}^{2}\right)
^{\frac{1}{2}}}\\
\leq&\disp{C_{7}(T+1)~~\mbox{for all}~~ T > 0.}\\
\end{array}
\label{ddfff5555ddffbnmbncz2ddfvgffgtyybhh.htt678ghhjjjddfghhhyuiihjj}
\end{equation}
This  readily establishes  \dref{1.1dddfgbhnjmdfgeddvbnmklllhyussddisda}.
\end{proof}

With the above compactness properties at hand, by means of a standard extraction procedure we can
now derive the following lemma which actually contains our main existence result already.

%
%
%
{\bf The proof of Theorem  \ref{theorem773}}
Firstly, 
in light of Lemmata \ref{lemmaghjssddgghhmk4563025xxhjklojjkkk} and \ref{qqqqlemma45630hhuujjuuyytt},
we conclude that there exists a positive constant $C_1$  such that
\begin{equation}
\begin{array}{rl}
\|u_{\varepsilon}\|_{L^{\frac{4}{3}}_{loc}([0,\infty); W^{1,\frac{4}{3}}(\Omega))}\leq C_1(T+1)~~~\mbox{and}~~~\|\partial_{t}u_{\varepsilon}\|_{L^{1}_{loc}([0,\infty); (W^{2,q}(\Omega))^*)}\leq C_1(T+1)
\end{array}
\label{ggjjssdffzcddffcfccvvfggvvvvgccvvvgjscz2.5297x963ccvbb111kkuu}
\end{equation}
as well as
\begin{equation}
\begin{array}{rl}
\|v_{\varepsilon}\|_{L^{2}_{loc}([0,\infty); W^{2,2}(\Omega))}\leq C_1(T+1)~~~\mbox{and}~~~\|\partial_{t}v_{\varepsilon}\|_{L^{2}_{loc}([0,\infty); (W^{1,2}(\Omega)))^*)}\leq C_1(T+1).
\end{array}
\label{ggjjssdffzcddffcfccvvfggvvddffvvgccvvvgjscz2.5297x963ccvbb111kkuu}
\end{equation}
Hence, collecting
\dref{ggjjssdffzcddffcfccvvfggvvvvgccvvvgjscz2.5297x963ccvbb111kkuu}--\dref{ggjjssdffzcddffcfccvvfggvvddffvvgccvvvgjscz2.5297x963ccvbb111kkuu} and  employing the the Aubin-Lions lemma (see e.g.
\cite{Simon}), we conclude that
\begin{equation}
\begin{array}{rl}
(u_{\varepsilon})_{\varepsilon\in(0,1)}~~~\mbox{is strongly precompact in}~~~L^{\frac{4}{3}}_{loc}(\bar{\Omega}\times[0,\infty)).
\end{array}
\label{ggjjssdffddfffzcccddffcfccvvfggvvggvvgccffghhvvvgjscz2.5297x963ccvbb111kkuu}
\end{equation}
as well as
\begin{equation}
\begin{array}{rl}
(v_{\varepsilon})_{\varepsilon\in(0,1)}~~~\mbox{is strongly precompact in}~~~L^{2}_{loc}(\bar{\Omega}\times[0,\infty)).
\end{array}
\label{ggjjssdffzcddddffcfcchhvvfggvvddffvvgccvvvgjscz2.5297x963ccvbb111kkuu}
\end{equation}
Therefore, there exists a subsequence $\varepsilon=\varepsilon_j\subset (0,1)_{j\in \mathbb{N}}$
and the limit functions $u,v$ and $w$
such that
\begin{equation}
u_{\varepsilon}\rightarrow u ~~\mbox{in}~~ L^{\frac{4}{3}}_{loc}(\bar{\Omega}\times[0,\infty))~~\mbox{and}~~\mbox{a.e.}~~\mbox{in}~~\Omega\times(0,\infty),
 \label{zjscz2.52ssdddd97x963ddfgh06662222tt3}
\end{equation}
\begin{equation}
v_{\varepsilon}\rightarrow v ~~\mbox{in}~~ L^{2}_{loc}(\bar{\Omega}\times[0,\infty))~~\mbox{and}~~\mbox{a.e.}~~\mbox{in}~~\Omega\times(0,\infty),
 \label{zjscz2.5297x963ddfgh06662222tt3}
\end{equation}
as well as
\begin{equation}
\nabla u_{\varepsilon}\rightharpoonup \nabla u~~\begin{array}{ll}
 \mbox{in}~~ L_{loc}^{\frac{4}{3}}(\bar{\Omega}\times[0,\infty))
   \end{array}\label{1.1ddfgghhhge666ccdf2345ddvbnmklllhyuisda}
\end{equation}
and
\begin{equation}
\Delta v_{\varepsilon}\rightharpoonup \Delta v~~\begin{array}{ll}
 \mbox{in}~~ L_{loc}^{2}(\bar{\Omega}\times[0,\infty)).
   \end{array}\label{1.1ddfgghhhsdddddfffge666ccdf2345ddvbnmklllhyuisda}
\end{equation}
Next, 
in light of  \dref{bnmbncz2.5ghhjuyuivvbssdddeennihjj},
%
%
there exists a subsequence $\varepsilon=\varepsilon_j\subset (0,1)_{j\in \mathbb{N}}$
such that
$\varepsilon_j\searrow 0$ as $j \rightarrow\infty$ 
\begin{equation}
 u_{\varepsilon}\rightharpoonup  u~~\begin{array}{ll}
\mbox{in}~~~L_{loc}^{2}(\bar{\Omega}\times[0,\infty)).\\
   \end{array}\label{1.1666ddccvvfggfgghhhgeccdf2345ddvbnmklllhyuisda}
\end{equation}
Next, let $g_{\varepsilon}(x, t) := -v_{\varepsilon}+u_{\varepsilon}.$
Therefore, recalling \dref{cz2.5ghju48cfg924ghyuji}
and \dref{bnmbncz2.5ghhjuyuivvbssdddeennihjj}, we conclude that $v_{\varepsilon t}-\Delta v_{\varepsilon} = g_{\varepsilon}$
is bounded in $L^{2} (\Omega\times(0, T))$ for any  $\varepsilon\in(0,1)$,  we may invoke the standard parabolic regularity theory  to infer that
$(v_{\varepsilon})_{\varepsilon\in(0,1)}$ is bounded in
$L^{2} ((0, T); W^{2,2}(\Omega))$.
Thus,  by \dref{wwwwwqqqq1.1dddfgbhnjmdfgeddvbnmklllhyussddisda} and the Aubin--Lions lemma we derive that  the relative compactness of $(v_{\varepsilon})_{\varepsilon\in(0,1)}$ in
$L^{2} ((0, T); W^{1,2}(\Omega))$. We can pick an appropriate subsequence which is
still written as $(\varepsilon_j )_{j\in \mathbb{N}}$ such that $\nabla v_{\varepsilon_j} \rightarrow z_1$
 in $L^{2} (\Omega\times(0, T))$ for all $T\in(0, \infty)$ and some
$z_1\in L^{2} (\Omega\times(0, T))$ as $j\rightarrow\infty$, hence $\nabla v_{\varepsilon_j} \rightarrow z_1$ a.e. in $\Omega\times(0, \infty)$
 as $j \rightarrow\infty$.
In view  of \dref{1.1ddfgghhhge666ccdf2345ddvbnmklllhyuisda} and  the Egorov theorem we conclude  that
$z_1=\nabla v,$ and whence
\begin{equation}
\nabla v\rightarrow \nabla v~~\begin{array}{ll}
  ~\mbox{a.e.}~~\mbox{in}~~\Omega\times(0,\infty)~~~\mbox{as}~~\varepsilon =\varepsilon_j\searrow0.
   \end{array}\label{1.1ddhhyujiiifgghhhge666ccdf2345ddvbnmklllhyuisda}
\end{equation}
In the following, we shall prove $(u,v)$ is a weak solution of problem \dref{1.1} in Definition \ref{df1}.
In fact, with the help of  \dref{zjscz2.5297x963ddfgh06662222tt3}--\dref{1.1666ddccvvfggfgghhhgeccdf2345ddvbnmklllhyuisda},
 we can derive  \dref{dffff1.1fghyuisdakkklll}.
Now, by the nonnegativity of $u_{\varepsilon}$ and $v_{\varepsilon}$, we derive  $u \geq 0$ and $v\geq 0$. 
On the other hand, in view of \dref{zjscz2.52ssdddd97x963ddfgh06662222tt3} and \dref{1.1ddhhyujiiifgghhhge666ccdf2345ddvbnmklllhyuisda}, we can infer from
\dref{1.1dddfgbhnjmdfgeddvbnmklllhyussddisda} that
$$
u_{\varepsilon}F_{\varepsilon}(u_{\varepsilon})\nabla v_{\varepsilon}\rightharpoonup z_2~~\begin{array}{ll}
  ~~~\mbox{in}~~ L^{1}(\Omega\times(0,T))~~\mbox{for each}~~ T \in(0, \infty).
   \end{array}
   $$
Next, due to \dref{1.ffggvbbnxxccvvn1}, \dref{zjscz2.52ssdddd97x963ddfgh06662222tt3} and \dref{1.1ddhhyujiiifgghhhge666ccdf2345ddvbnmklllhyuisda}, we derive that
\begin{equation}u_{\varepsilon} F_{\varepsilon}(u_{\varepsilon})\nabla v_{\varepsilon}\rightarrow u\nabla v~~~\mbox{a.e.}~~\mbox{in}~~\Omega\times(0,\infty)~~~\mbox{as}~~\varepsilon =\varepsilon_j\searrow0.
\label{1.1ddddfddffttyygghhyujiiifgghhhgffgge666ccdf2345ddvbnmklllhyuisda}
\end{equation}
Therefore, by the Egorov theorem, we can get $z_2= u\nabla v,$ and hence
\begin{equation}
u_{\varepsilon}F_{\varepsilon}(u_{\varepsilon})\nabla v_{\varepsilon}\rightharpoonup u\nabla v~~\begin{array}{ll}
  ~~~\mbox{in}~~ L^{1}(\Omega\times(0,T))~~\mbox{for each}~~ T \in(0, \infty).
   \end{array}\label{zxxcvvfgggjscddf4556ffg77ffcvvfggz2.5ty}
\end{equation}
Therefore, by \dref{1.1ddfgghhhge666ccdf2345ddvbnmklllhyuisda} and  \dref{zxxcvvfgggjscddf4556ffg77ffcvvfggz2.5ty}, we conclude that the integrability of
$\nabla u$ and $u\nabla v$  in \dref{726291hh}.
 Finally, according to \dref{zjscz2.5297x963ddfgh06662222tt3}--\dref{1.1ddfgghhhsdddddfffge666ccdf2345ddvbnmklllhyuisda} and \dref{zxxcvvfgggjscddf4556ffg77ffcvvfggz2.5ty}, we may pass to the limit in
the respective weak formulations associated with the the regularized system \dref{1.1} and get
 the integral
identities \dref{eqx45xx12112ccgghh}--\dref{eqx45xx12112ccgghhjj}.

\section{The boundedness  and classical  solution of \dref{1.1}}

In order to discuss the boundedness  and classical  solution of \dref{1.1}, firstly,  we will
recall the known result about local existence of solutions to \dref{1.1} (see the proof of Lemma 1.1 of \cite{Winkler37103}).
\begin{lemma}\label{1118lemma70}Let $\Omega\subset\mathbb{R}^N(N\geq1)$ be a smooth bounded  domain.
Assume that the nonnegative functions $u_0\in C^{0}(\bar{\Omega})$ and $v_0\in W^{1,\theta}(\bar{\Omega})$ (with some $\theta> N$).
Then there exist a
maximal $T_{max}\in(0,\infty]$ and a uniquely determined pair $(u,v)$ of nonnegative functions
$$
 \left\{\begin{array}{ll}
 u\in C^0(\bar{\Omega}\times[0,T_{max}))\cap C^{2,1}(\bar{\Omega}\times(0,T_{max})),\\
  v\in  C^0(\bar{\Omega}\times[0,T_{max}))\cap C^{2,1}(\bar{\Omega}\times(0,T_{max}))\cap
   L^\infty_{loc}((0,T_{max}); W^{1,\theta}(\Omega))\\
   \end{array}\right.
   $$
that solve \dref{1.1} in the classical sense in
 $\Omega\times(0, T_{max})$. 
%
%
%
Moreover, if  $T_{max}<+\infty$, then
\begin{equation}
\|u(\cdot, t)\|_{L^\infty(\Omega)}\rightarrow\infty~~ \mbox{as}~~ t\nearrow T_{max}
\label{11181.163072x}
\end{equation}
is fulfilled.
\end{lemma}

The following result is similar to Lemma 3.4 of \cite{Zheng33312186}, which plays an important role
in the proof of Theorem \ref{theorem3}.
\begin{lemma}\label{lemma45630223116}
Let \begin{equation}
{A}_1=\frac{1}{ \delta+1}\left[\frac{ \delta+1}{ \delta}\right]^{- \delta }\left(\frac{\delta-1}{\delta} \right)^{ \delta+1}
\label{zjscz2.5297x9630222211444125}
\end{equation}
and $H(y)=y+{A}_1y^{- \delta }\chi^{ \delta+1}C_{ \delta+1}$ for $y>0.$
For any fixed $\delta\geq1,\chi,C_{\delta+1}>0,$
Then $$\min_{y>0}H(y)=\frac{(\delta-1)}{\delta}C_{\delta+1}^{\frac{1}{\delta+1}}\chi.$$
\end{lemma}
\begin{proof}
It is easy to verify that $$H'(y)=1- A_1\delta C_{\delta+1}\left(\frac{\chi }{y} \right)^{\delta+1}.$$
Let $H'(y)=0$, we have
$$y=\left(A_1C_{\delta+1}\delta\right)^{\frac{1}{\delta+1}}\chi.$$
On the other hand, by $\lim_{y\rightarrow0^+}H(y)=+\infty$ and $\lim_{y\rightarrow+\infty}H(y)=+\infty$, we have
$$\begin{array}{rl}
\min_{y>0}H(y)=H[\left(A_1C_{\delta+1}\delta\right)^{\frac{1}{\delta+1}}\chi]=&\disp{\left(A_1C_{\delta+1}\right)^{\frac{1}{\delta+1}}
(\delta^{\frac{1}{\delta+1}}+\delta^{-\frac{\delta}{\delta+1}})\chi}\\
=&\disp{\frac{(\delta-1)}{\delta}C_{\delta+1}^{\frac{1}{\delta+1}}\chi.}\\
\end{array}
$$
\end{proof}
In order to discuss the boundedness  and classical  solution of \dref{1.1}, in light of Lemma \ref{1118lemma70},
firstly,  let us  pick any $s_0\in(0,T_{max})$ and $s_0\leq1$, 
 there exists
$K>0$ such that
\begin{equation}\label{eqx45xx12112}
\|u(\tau)\|_{L^\infty(\Omega)}\leq K,~~\|v(\tau)\|_{L^\infty(\Omega)}\leq K~~\mbox{and}~~\|\Delta v(\tau)\|_{L^\infty(\Omega)}\leq K~~\mbox{for all}~~\tau\in[0,s_0].
\end{equation}
\begin{lemma}\label{lemma45630223} Let $\Omega\subset\mathbb{R}^N(N\geq1)$ be a smooth bounded  domain.
Assume that    $\mu>\frac{(N-2)_+}{N}\chi C_{\frac{N}{2}+1}^{\frac{1}{\frac{N}{2}+1}}$, where $C_{\frac{N}{2}+1}$ is given by Lemma \ref{lemma45xy1222232}
(with $\gamma=\frac{N}{2}+1$ in Lemma \ref{lemma45xy1222232}).
    Let $(u,v)$ be a solution to \dref{1.1} on $(0,T_{max})$.
 Then 
 for all $p>1$,
there exists a positive constant $C:=C(p,|\Omega|,\mu,\chi,K)$ such that 
\begin{equation}
\int_{\Omega}u^p(x,t)\leq C ~~~\mbox{for all}~~ t\in(0,T_{max}).
\label{zjscz2.5297x96302222114}
\end{equation}
\end{lemma}
\begin{proof}
Multiplying the first equation of \dref{1.1}
  by $u^{{r}-1}$ and integrating over $\Omega$, 
 we get
\begin{equation}
\begin{array}{rl}
&\disp{\frac{1}{{r}}\frac{d}{dt}\|u\|^{{r}}_{L^{{r}}(\Omega)}+({{r}-1})\int_{\Omega}u^{{{r}-2}}|\nabla u|^2}
\\
=&\disp{-\chi\int_\Omega \nabla\cdot( u\nabla v)
  u^{{r}-1} +
\int_\Omega   u^{{r}-1}(au-\mu u^2)~~\mbox{for all}~~ t\in(0, T_{max}) ,}\\
\end{array}
\label{cz2.5114114}
\end{equation}
that is,
\begin{equation}
\begin{array}{rl}
&\disp{\frac{1}{{r}}\frac{d}{dt}\|u\|^{{{r}}}_{L^{{r}}(\Omega)}}
\\
\leq&\disp{-\frac{{r}+1}{{r}}\int_{\Omega} u^{r} -\chi\int_\Omega \nabla\cdot( u\nabla v)
  u^{{r}-1} }\\
 &+\disp{\int_\Omega \left(\frac{{r}+1}{{r}} u^{r}+  u^{{r}-1}(au-\mu u^2)\right)
 ~~\mbox{for all}~~ t\in(0, T_{max}).}\\
\end{array}
\label{cz2.5kk1214114114}
\end{equation}
Hence, by the Young inequality, it reads that
\begin{equation}
\begin{array}{rl}
&\disp{\int_\Omega  \left(\frac{{r}+1}{{r}} u^{r}+ u^{{r}-1}(au-\mu u^2)\right) }\\
\leq &\disp{\frac{{r}+1}{{r}}\int_\Omega u^{r} +a\int_\Omega    u^{r}- \mu\int_\Omega  u^{{{r}+1}}}\\
\leq &\disp{(\varepsilon_1- \mu)\int_\Omega u^{{{r}+1}} +C_1(\varepsilon_1,{r}),}
\end{array}
\label{cz2.563011228ddff}
\end{equation}
where $$C_1(\varepsilon_1,{r})=\frac{1}{{r}+1}\left(\varepsilon_1\frac{{r}+1}{{r}}\right)^{-{r} }
\left(\frac{{r}+1}{{r}}+a\right)^{{r}+1 }|\Omega|.$$
%
%

Next,
integrating by parts to the first term on the right hand side of \dref{cz2.5114114},  using  the Young inequality and \dref{1.ffggvbbnxxccvvn1},
we obtain 
\begin{equation}
\begin{array}{rl}
&\disp{-\chi\int_\Omega \nabla\cdot( u\nabla v)u^{{r}-1} }
\\
=&\disp{({{r}-1})\chi\int_\Omega  u^{{{r}-1}}\nabla u\cdot\nabla v }
\\
=&\disp{-\frac{{{r}-1}}{{r}}\chi \int_\Omega u^{{r}}\Delta v }
\\
\leq&\disp{\frac{{{r}-1}}{{r}}\chi \int_\Omega u^{r}|\Delta v|~~\mbox{for all}~~ t\in(0, T_{max}). }
\\
\end{array}
\label{cz2.563019114}
\end{equation}
Now, let \begin{equation}
\lambda_0:=\left(A_1C_{{r}+1}{r}\right)^{\frac{1}{{r}+1}}\chi,
\label{cz2.56ssd30er191rrrr14}
\end{equation}
where ${A}_1$ is given by \dref{zjscz2.5297x9630222211444125}.
While from \dref{cz2.563019114} and the Young inequality, we
have
\begin{equation}
\begin{array}{rl}
&\disp{-\chi\int_\Omega \nabla\cdot( u\nabla v)u^{{r}-1}  }
\\
\leq&\disp{\lambda_0\int_\Omega  u^{{r}+1}+\frac{1}{ {r}+1}\left[\lambda_0\frac{ {r}+1}{ {r}}\right]^{- {r} }\left(\frac{{r}-1}{{r}}\chi \right)^{ {r}+1}\int_\Omega |\Delta v|^{ {r}+1} }
\\
=&\disp{\lambda_0\int_\Omega  u^{{r}+1}+{A}_1\lambda_0^{- {r} }\chi^{ {r}+1}\int_\Omega |\Delta v|^{ {r}+1} ~~\mbox{for all}~~ t\in(0, T_{max}).}
\\
\end{array}
\label{cz2.563019114gghh}
\end{equation}
Thus, inserting \dref{cz2.563011228ddff} and \dref{cz2.563019114gghh} into \dref{cz2.5kk1214114114}, we get
\begin{equation*}
\begin{array}{rl}
\disp\frac{1}{{r}}\disp\frac{d}{dt}\|u\|^{{{r}}}_{L^{{r}}(\Omega)}\leq&\disp{(\varepsilon_1+\lambda_0- \mu)\int_\Omega u^{{{r}+1}} -\frac{{r}+1}{{r}}\int_{\Omega} u^{r} }\\
&+\disp{{A}_1\lambda_0^{- {r} }\chi^{ {r}+1}\int_\Omega |\Delta v|^{ {r}+1} +
C_1(\varepsilon_1,{r})~~\mbox{for all}~~ t\in(0, T_{max}).}\\
\end{array}
\end{equation*}
For any $t\in (s_0,T_{max})$,
employing the variation-of-constants formula to the above inequality, we obtain
\begin{equation}
\begin{array}{rl}
&\disp{\frac{1}{{r}}\|u(t) \|^{{{r}}}_{L^{{r}}(\Omega)}}
\\
\leq&\disp{\frac{1}{{r}}e^{-( {r}+1)(t-s_0)}\|u(s_0) \|^{{{r}}}_{L^{{r}}(\Omega)}+(\varepsilon_1+\lambda_0- \mu)\int_{s_0}^t
e^{-( {r}+1)(t-s)}\int_\Omega u^{{{r}+1}}}\\
&+\disp{{A}_1\lambda_0^{- {r} }\chi^{ {r}+1}\int_{s_0}^t
e^{-( {r}+1)(t-s)}\int_\Omega |\Delta v|^{ {r}+1} + C_1(\varepsilon_1,{r})\int_{s_0}^t
e^{-( {r}+1)(t-s)}}\\
\leq&\disp{(\varepsilon_1+\lambda_0- \mu)\int_{s_0}^t
e^{-( {r}+1)(t-s)}\int_\Omega u^{{{r}+1}}}\\
&+\disp{{A}_1\lambda_0^{- {r} }\chi^{ {r}+1}\int_{s_0}^t
e^{-( {r}+1)(t-s)}\int_\Omega |\Delta v|^{ {r}+1}+C_2({r},\varepsilon_1),}\\
\end{array}
\label{cz2.5kk1214114114rrgg}
\end{equation}
where
$$C_2:=C_2({r},\varepsilon_1)=\frac{1}{{r}}\|u(s_0) \|^{{{r}}}_{L^{{r}}(\Omega)}+
 C_1(\varepsilon_1,{r})\int_{s_0}^t
e^{-( {r}+1)(t-s)}ds$$
and $s_0$ is the same as \dref{eqx45xx12112}.

Now, by Lemma \ref{lemma45xy1222232}, we have
\begin{equation}\label{cz2.5kk1214114114rrggjjkk}
\begin{array}{rl}
&\disp{{A}_1\lambda_0^{- {r} }\chi^{ {r}+1}\int_{s_0}^t
e^{-( {r}+1)(t-s)}\int_\Omega |\Delta v|^{ {r}+1}}
\\
=&\disp{{A}_1\lambda_0^{- {r} }\chi^{ {r}+1}e^{-( {r}+1)t}\int_{s_0}^t
e^{( {r}+1)s}\int_\Omega |\Delta v|^{ {r}+1} }\\
\leq&\disp{{A}_1\lambda_0^{- {r} }\chi^{ {r}+1}e^{-( {r}+1)t}C_{ {r}+1}[\int_{s_0}^t
\int_\Omega e^{( {r}+1)s}u^{ {r}+1} }\\
&+\disp{e^{( {r}+1)s_0}(\|v(\cdot,s_0)\|^{{r}+1}_{L^{{r}+1}(\Omega)}+\|\Delta v(\cdot,s_0)\|^{{r}+1}_{L^{{r}+1}(\Omega)})]}\\
\end{array}
\end{equation}
for all $t\in(s_0, T_{max})$.
By substituting \dref{cz2.5kk1214114114rrggjjkk} into \dref{cz2.5kk1214114114rrgg}, using \dref{cz2.56ssd30er191rrrr14} and Lemma \ref{lemma45630223116}, we get
\begin{equation}
\begin{array}{rl}
&\disp{\frac{1}{{r}}\|u(t) \|^{{{r}}}_{L^{{r}}(\Omega)}}
\\
\leq&\disp{(\varepsilon_1+\lambda_0+{A}_1\lambda_0^{- {r} }\chi^{ {r}+1}C_{ {r}+1}- \mu)\int_{s_0}^t
e^{-( {r}+1)(t-s)}\int_\Omega u^{{{r}+1}} }\\
&+\disp{{A}_1\lambda_0^{- {r} }\chi^{ {r}+1}e^{-( {r}+1)(t-s_0)}C_{ {r}+1}(\|v(\cdot,s_0)\|^{{r}+1}_{L^{{r}+1}(\Omega)}+\|\Delta v(\cdot,s_0)\|^{{r}+1}_{L^{{r}+1}(\Omega)})+C_2({r},
\varepsilon_1)}\\
=&\disp{(\varepsilon_1+\frac{({r}-1)}{{r}}C_{{r}+1}^{\frac{1}{{r}+1}}\chi- \mu)\int_{s_0}^t
e^{-( {r}+1)(t-s)}\int_\Omega u^{{{r}+1}} }\\
&+\disp{{A}_1\lambda_0^{- {r} }\chi^{ {r}+1}e^{-( {r}+1)(t-s_0)}C_{ {r}+1}(\|v(\cdot,s_0)\|^{{r}+1}_{L^{{r}+1}(\Omega)}+\|\Delta v(\cdot,s_0)\|^{{r}+1}_{L^{{r}+1}(\Omega)})+C_2({r},
\varepsilon_1).}\\
\end{array}
\label{cz2.5kk1214114114rrggkkll}
\end{equation}
Since, $\mu>\frac{(N-2)_+}{N}\chi C_{\frac{N}{2}+1}^{\frac{1}{\frac{N}{2}+1}}$, we may choose  $r:={q_0}>\frac{N}{2}$ in \dref{cz2.5kk1214114114rrggkkll}
such that
$$\mu>\frac{{q_0}-1}{{q_0}}\chi C_{{q_0}+1}^{\frac{1}{{q_0}+1}},$$
thus, pick $\varepsilon_1$ appropriating  small such that
$$0<\varepsilon_1<\mu -\frac{{q_0}-1}{{q_0}}\chi C_{{q_0}+1}^{\frac{1}{{q_0}+1}},$$
then in light of \dref{cz2.5kk1214114114rrggkkll}, we derive that there exists a positive constant $C_3$
such that
\begin{equation}
\begin{array}{rl}
&\disp{\int_{\Omega}u^{{q_0}}(x,t) dx\leq C_3~~\mbox{for all}~~t\in (s_0, T_{max}).}\\
\end{array}
\label{cz2.5kk1214114114rrggkklljjuu}
\end{equation}
Next,
we fix $q <\frac{N{q_0}}{(N-{q_0})^+}$
and choose some
 $\alpha> \frac{1}{2}$ such that
\begin{equation}
q <\frac{1}{\frac{1}{q_0}-\frac{1}{N}+\frac{2}{N}(\alpha-\frac{1}{2})}\leq\frac{N{q_0}}{(N-{q_0})^+}.
\label{fghgbhnjcz2.5ghju48cfg924ghyuji}
\end{equation}

Now, involving the variation-of-constants formula
for $v$, we have
\begin{equation}
v(t)=e^{-\tau(A+1)}v(s_0) +\int_{s_0}^{t}e^{-(t-s)(A+1)}u(s) ds,~~ t\in(s_0, T_{max}),
\label{fghbnmcz2.5ghju48cfg924ghyuji}
\end{equation}
where $A := A_p$ denote the sectorial operator defined by
$$
A_pu :=-\Delta u~~\mbox{for all}~~u\in D(A_p) :=\{\varphi\in W^{2,p}(\Omega)|\frac{\partial\varphi}{\partial \nu}|_{\partial\Omega}=0\}.
$$
Hence, it follows from \dref{eqx45xx12112} and  
 \dref{fghbnmcz2.5ghju48cfg924ghyuji} that
\begin{equation}
\begin{array}{rl}
&\disp{\|(A+1)^\alpha v(t)\|_{L^q(\Omega)}}\\
\leq&\disp{C_4\int_{s_0}^{t}(t-s)^{-\alpha-\frac{N}{2}(\frac{1}{q_0}-\frac{1}{q})}e^{-\mu(t-s)}\|u(s)\|_{L^{q_0}(\Omega)}ds+
C_4s_0^{-\alpha-\frac{N}{2}(1-\frac{1}{q})}\|v(s_0,t)\|_{L^1(\Omega)}}\\
\leq&\disp{C_4\int_{0}^{+\infty}\sigma^{-\alpha-\frac{N}{2}(\frac{1}{q_0}-\frac{1}{q})}e^{-\mu\sigma}d\sigma
+C_4s_0^{-\alpha-\frac{N}{2}(1-\frac{1}{q})}K,}\\
\end{array}
\label{gnhmkfghbnmcz2.5ghju48cfg924ghyuji}
\end{equation}
where  $s_0$ is the same as \dref{eqx45xx12112}.
Hence, due to \dref{fghgbhnjcz2.5ghju48cfg924ghyuji}  and \dref{gnhmkfghbnmcz2.5ghju48cfg924ghyuji}, we have
\begin{equation}
\int_{\Omega}|\nabla {v}(t)|^{q}\leq C_5~~\mbox{for all}~~ t\in(s_0, T_{max})
\label{ffgbbcz2.5ghju48cfg924ghyuji}
\end{equation}
and $q\in[1,\frac{N{q_0}}{(N-{q_0})^+})$.
Finally, in view of \dref{eqx45xx12112} and \dref{ffgbbcz2.5ghju48cfg924ghyuji},
 we can get \begin{equation}
\int_{\Omega}|\nabla {v}(t)|^{q}\leq C_6~~\mbox{for all}~~ t\in(0, T_{max})~~\mbox{and}~~q\in[1,\frac{N{q_0}}{(N-{q_0})^+}).
\label{ffgbbcz2.5ghjusseeeddd48cfg924ghyuji}
\end{equation}
with some positive constant $C_6.$

Multiplying both sides of the first equation in \dref{1.1} by $u^{p-1}$, integrating over $\Omega$ and integrating by parts, we arrive at
\begin{equation}
\begin{array}{rl}
&\disp{\frac{1}{{p}}\frac{d}{dt}\|u\|^{{p}}_{L^{{p}}(\Omega)}+({{p}-1})\int_{\Omega}u^{{{p}-2}}|\nabla u|^2}
\\
=&\disp{-\chi\int_\Omega \nabla\cdot( u\nabla v)
  u^{{p}-1} +
\int_\Omega   u^{{p}-1}(au-\mu u^2) }\\
=&\disp{\chi({p}-1)\int_\Omega  u^{{p}-1}\nabla u\cdot\nabla v
   +
\int_\Omega   u^{{p}-1}(au-\mu u^2) ,}\\
\end{array}
\label{cz2aasweee.5114114}
\end{equation}
which together with the Young inequality and \dref{1.ffggvbbnxxccvvn1} implies that
\begin{equation}
\begin{array}{rl}
&\disp{\frac{1}{{p}}\frac{d}{dt}\|u\|^{{p}}_{L^{{p}}(\Omega)}+({{p}-1})\int_{\Omega}u^{{{p}-2}}|\nabla u|^2}
\\
\leq&\disp{\frac{{{p}-1}}{2}\int_{\Omega}u^{{{p}-2}}|\nabla u|^2+
\frac{\chi^2({p}-1)}{2}\int_\Omega  u^{{p}}|\nabla v|^2
   -\frac{\mu }{2}\int_\Omega u^{p+1}+ C_7}\\
\end{array}
\label{cz2aasweee.5ssedfff114114}
\end{equation}
for some positive constant $C_7.$
Since, $q_0>\frac{N}{2}$ yields $q_0<\frac{N{q_0}}{2(N-{q_0})^+}$, in light of the H\"{o}lder inequality and \dref{ffgbbcz2.5ghjusseeeddd48cfg924ghyuji}, we derive   at
\begin{equation}
\begin{array}{rl}
 \disp\frac{\chi^2({p}-1)}{2}\disp\int_\Omega{{u^{p }}} |\nabla v|^2\leq&\disp{ \disp\frac{\chi^2({p}-1)}{2}\left(\disp\int_\Omega{{u^{\frac{q_0}{q_0-1} p }}}\right)^{\frac{q_0-1}{q_0}}\left(\disp\int_\Omega |\nabla {v}|^{2q_0}\right)^{\frac{1}{q_0}}}\\
\leq&\disp{C_8\|  {{u^{\frac{p}{2}}}}\|^{2}_{L^{2\frac{q_0}{q_0-1} }(\Omega)},}\\
\end{array}
\label{cz2.57151hhkkhhhjukildrfthjjhhhhh}
\end{equation}
where $C_8$ is a positive constant.
Since 
$q_0> \frac{N}{2}$ and $p>q_0-1$,
we have
$$\frac{q_0}{p}\leq\frac{q_0}{q_0-1}\leq\frac{N}{N-2},$$
which together with the Gagliardo--Nirenberg inequality
 implies that
\begin{equation}
\begin{array}{rl}
C_8\|  {{u^{\frac{p}{2}}}}\|
^{2}_{L^{2\frac{q_0}{q_0-1} }(\Omega)}\leq&\disp{C_9(\|\nabla   {{u^{\frac{p}{2}}}}\|_{L^2(\Omega)}^{\mu_1}\|  {{u^{\frac{p}{2}}}}\|_{L^\frac{2q_0}{p}(\Omega)}^{1-\mu_1}+\|  {{u^{\frac{p}{2}}}}\|_{L^\frac{2q_0}{p}(\Omega)})^{2}}\\
\leq&\disp{C_{10}(\|\nabla   {{u^{\frac{p}{2}}}}\|_{L^2(\Omega)}^{2\mu_1}+1)}\\
=&\disp{C_{10}(\|\nabla   {{u^{\frac{p}{2}}}}\|_{L^2(\Omega)}^{\frac{2N(p-q_0+1)}{Np+2q_0-Nq_0}}+1)}\\
\end{array}
\label{cz2.563022222ikopl2sdfg44}
\end{equation}
with some positive constants $C_9, C_{10}$ and
$$\mu_1=\frac{\frac{N{p}}{2q_0}-\frac{Np}{2\frac{q_0}{q_0-1} p }}{1-\frac{N}{2}+\frac{N{p}}{2q_0}}=
{p}\frac{\frac{N}{2q_0}-\frac{N}{2\frac{q_0}{q_0-1} p }}{1-\frac{N}{2}+\frac{N{p}}{2q_0}}\in(0,1).$$
Now, in view of the Young inequality, we derive that
\begin{equation}
\begin{array}{rl}
\disp\frac{\chi^2({p}-1)}{2}\disp\int_\Omega  u^{{p}}|\nabla v|^2 &\leq\disp{\frac{{{p}-1}}{4}\int_{\Omega}u^{{{p}-2}}
|\nabla u|^2+C_{11}.}\\
\end{array}
\label{cz2aasweeeddfff.5ssedfssddff114114}
\end{equation}
Inserting \dref{cz2aasweeeddfff.5ssedfssddff114114} into \dref{cz2aasweee.5ssedfff114114}, we conclude that
\begin{equation}
\begin{array}{rl}
&\disp{\frac{1}{{p}}\frac{d}{dt}\|u\|^{{p}}_{L^{{p}}(\Omega)}+\frac{{{p}-1}}{4}\int_{\Omega}u^{{{p}-2}}|\nabla u|^2+
\frac{\mu}{2}\int_\Omega u^{p+1}\leq C_{12}.}\\
\end{array}
\label{cz2aasweee.5ssedfff114114}
\end{equation}
Therefore, letting
  $y:=\disp\int_{\Omega}u^{p}$ 
in \dref{cz2aasweee.5ssedfff114114} yields to
$$\frac{d}{dt}y(t)+C_{13}y^h(t)\leq C_{14}~~\mbox{for all}~~ t\in(0,T_{max})$$
with some positive constant $h$.
 Thus a standard ODE comparison argument implies
\begin{equation}
\begin{array}{rl}
\|u(\cdot, t)\|_{L^{{p}}(\Omega)}\leq C_{15} ~~ \mbox{for all}~~p\geq1~~\mbox{and}~~  t\in(0,T_{max}) \\
\end{array}
\label{cz2.5g556789hhjui78jj90099}
\end{equation}
for some positive constant $C_{15}$.
%
%
The proof Lemma \ref{lemma45630223} is completed.
\end{proof}
Our main result on global existence and boundedness thereby becomes a straightforward consequence
of Lemma \ref{1118lemma70} and  Lemma \ref{lemma45630223}.

{\bf The proof of Theorem \ref{theorem3}}~
Theorem \ref{theorem3} will be proved if we can show $T_{max}=\infty$. Suppose on contrary that $T_{max}<\infty$.
Due to
$\|u(\cdot, t)\|_{L^p(\Omega)}$ is bounded for any large $p$, we infer from the fundamental estimates for Neumann semigroup
(see Lemma 4.1 of \cite{Horstmann791}) or the standard regularity theory of parabolic equation (see e.g. Ladyzenskaja et al.  \cite{Ladyzenskaja710}) that
\begin{equation}
\begin{array}{rl}
\|\nabla v(\cdot, t)\|_{L^{{\infty}}(\Omega)}\leq C_{1} ~~ \mbox{for all}~~~  t\in(0,T_{max}) \\
\end{array}
\label{cz2sedfgg.5g5gghh56789hhjui7ssddd8jj90099}
\end{equation}
and some positive constant $C_1.$

%
Upon an application of the well-known Moser-Alikakos iteration procedure (see Lemma A.1 in \cite{Tao794}),
we see that
\begin{equation}
\begin{array}{rl}
\|u(\cdot, t)\|_{L^{{\infty}}(\Omega)}\leq C_{2} ~~ \mbox{for all}~~~  t\in(0,T_{max}) \\
\end{array}
\label{cz2.5g5gghh56789hhjui78jj90099}
\end{equation}
and a positive constant $C_{2}$.

In view of \dref{cz2sedfgg.5g5gghh56789hhjui7ssddd8jj90099} and \dref{cz2.5g5gghh56789hhjui78jj90099}, we apply Lemma \ref{1118lemma70} to reach a contradiction.
 Hence the  classical solution $(u,v)$ of \dref{1.1} is global in time and bounded. Finally, employing  the same arguments as in the proof of Lemma 1.1 in \cite{Winkler37103}, and taking advantage of \dref{cz2.5g5gghh56789hhjui78jj90099}, we conclude the uniqueness of solution to \dref{1.1}.

 \section{Decay. Proof of Theorem \ref{theoremfgb3}}
In this section we study the long-time behavior for \dref{1.1} in the case $a = 0$.
As the first step, we give the decay property separately for the integrals of the solution components
$u$ and $v$.
\begin{lemma}\label{fghfbglemma4563025xxhjklojjkkkgyhuissddff}
Let $\Omega\subset\mathbb{R}^N(N\geq1)$ be a smooth bounded  domain.
Assume that $a= 0$. 
 Then we have
\begin{equation}
\begin{array}{rl}\label{gbhnggeqx45xx1211}
&\disp{\int_{\Omega}{{u}}(x,t)\leq\left(\left(\int_{\Omega}{u}_0(x)\right)^{-1}+\mu|\Omega|^{-1}t\right)^{{-1}}~~~\mbox{for all}~~t\in(0,\infty).}
\\
\end{array}
\end{equation}
\end{lemma}
\begin{proof}
Let $t > 0$ and $s\in(0, t)$. Since $a= 0$, it follows from an integration by parts to the first
equation in \dref{1.1}
and the H\"{o}lder inequality that
\begin{equation}
\begin{array}{rl}\label{jhjeqx45xx1211}
&\disp{\frac{d}{ds}\int_{\Omega}{{u}}(x,s) =-\mu\int_{\Omega}{{u}^{2}}(x,s)\leq-\mu|\Omega|^{-1}\left(\int_{\Omega}{{u}}(x,s)\right)^{2}~~\mbox{for all}~~ s\in(0, t),}
\\
\end{array}
\end{equation}
which implies that
\begin{equation}
\begin{array}{rl}\label{ggeqx45xx1211}
&\disp{\left(\int_{\Omega}{{u}}(x,t)\right)^{-1}-\left(\int_{\Omega}{u}_0(x)\right)^{-1}\geq\mu|\Omega|^{-1}t,}
\\
\end{array}
\end{equation}
which in light of \dref{ggeqx45xx1211} implies that  \dref{gbhnggeqx45xx1211}  holds.
\end{proof}

As a consequence, we obtain a basic decay property also for the second solution component.
\begin{lemma}\label{fhhghfbglemma4563025xxhjklojjkkkgyhuissddff}
Let $\Omega\subset\mathbb{R}^N(N\geq1)$ be a smooth bounded  domain.
There exists $C > 0$ such that
\begin{equation}\label{kmlgbhnggeqx45xx1211}
\int_{\Omega}{{v}}(x,t)\leq\frac{C}{(1+t)}~~~\mbox{for all}~~ t\in(0, T_{max}).
\end{equation}
\end{lemma}

\begin{proof}
Let $z(t) :=\int_{\Omega}
 {{v}}(x, t)$ for $t\in[0, T_{max})$.
Then integrating the second equation in \dref{1.1},  we conclude  that there exists a positive constant $C_1 $
 such that
 \begin{equation}
\begin{array}{rl}\label{hhhjhjeqx45xx1211}
z'(t)=&\disp{-z(t)+\int_{\Omega}u(x,t)}
\\
\leq&\disp{-z(t)+C_1(1+t)^{{-1}}  ~~~\mbox{for all}~~ t\in(0, T_{max}).}\\
\end{array}
\end{equation}
 Put
$C_2 :=2\max\{\int_{\Omega}v_0(x),2\}$
and define
$$\bar{z}(t) :=C_2(t + 2)^{-1}~~~\mbox{for all}~~t\geq0.$$
Then $\bar{z}(0) = \frac{C_2}{2} \geq\int_{\Omega}v_0(x)=z_0$ and
 \begin{equation}
 \begin{array}{rl}\label{gbhnhhhjhjeqx45xx1211}
&\disp{\bar{z}'(t)+\bar{z}(t) -C_1(1+t)^{{-1}}}\\
 =&\disp{-C_2(t + 2)^{-2}
+C_2(t + 2)^{-1}-C_1(1+t)^{{-1}}}
\\
\geq&\disp{C_2(t + 2)^{-2}(\frac{1}{2}-\frac{1}{t + 2})+\frac{1}{2}(t + 2)^{-1}(C_2-2C_1
(t + 2)(1+t)^{{-1}})}\\
\geq&\disp{C_2(t + 2)^{-2}(\frac{1}{2}-\frac{1}{ 2})+\frac{1}{2}(t + 2)^{-1}[C_2-C_12^{2}
]}\\
\geq&\disp{0~~~\mbox{for all}~~ t>0.}\\
\end{array}
\end{equation}
With the help of the comparison, we thus infer that
 $z(t)\leq \bar{z}(t)$ for all $t\in(0, T_{max})$, which directly establishes \dref{kmlgbhnggeqx45xx1211}.
\end{proof}

In turning the basic decay information on $u$ from Lemma \ref{fghfbglemma4563025xxhjklojjkkkgyhuissddff} into the uniform convergence property
asserted in Theorem \ref{theoremfgb3}, we shall make use of the following H\"{o}lder estimate implied by the regularity
properties collected in the previous section.


\begin{lemma}\label{ffgghhfhhghfbglemma4563025xxhjklojjkkkgyhuissddff}
Let $\Omega\subset\mathbb{R}^N(N\geq1)$ be a smooth bounded  domain.
Let $a= 0$ and $\mu>\frac{(N-2)_+}{N}\chi C_{\frac{N}{2}+1}^{\frac{1}{\frac{N}{2}+1}}$.
Then with $(u, v)$ as given by Theorem \ref{theorem3}, we have
\begin{equation}\label{hnjgghhgbhnggeqx45xx1211}
\|u\|_{C^{\theta,\frac{\theta}{2}}(\bar{\Omega}\times[t,t+1))}\leq C~~ \mbox{for all}~~ t>1
\end{equation}
and some positive constant $\theta\in (0, 1)$ and $C > 0$.
\end{lemma}
\begin{proof}
Firstly,
rewriting the first equation of \dref{1.1} in the form
\begin{equation}
\begin{array}{rl}\label{hhhssdddjhjeqsddffx45xx1211}
u_t=&\disp{\nabla\cdot(\nabla u-h_1(x,t))+h_2(x,t),~~x\in \Omega, t>0,}
\\
\end{array}
\end{equation}
where
$$h_1(x,t):=u(x,t)\nabla v(x,t)~~~\mbox{and}~~h_2(x,t):=au(x,t)-\mu u^2(x,t) $$
for $x\in \Omega$ and $t > 0.$
On the other hand, in view of  Theorem \ref{theorem3},
\begin{equation}\label{fftggghhhssdddjhjeqsddffx45xx1211}
\mbox{both}~~ h_1~~ \mbox{and}~~ h_2~~\mbox{are bounded in}~~
L^\infty((0,\infty);L^q(\Omega))~~ \mbox{for any}~~ q\in(1,\infty).
\end{equation}
Next,
by the Young inequality, we derive from \dref{hhhssdddjhjeqsddffx45xx1211} that
\begin{equation}\label{fftggghhssderrrhssdddjhjeqsddffx45xx1211}(\nabla u-h_1)\cdot \nabla u\geq \frac{1}{2}|\nabla u|^2-\frac{1}{2}|h_1|^2~~~\mbox{and}~~|\nabla u-h_1|\leq|\nabla u|+|h_1|
~~\mbox{in}~~\Omega\times(0,\infty).
 \end{equation}
Collecting \dref{fftggghhhssdddjhjeqsddffx45xx1211}--\dref{fftggghhssderrrhssdddjhjeqsddffx45xx1211},  applying the standard regularity theory of parabolic equation, we may conclude from \dref{hhhssdddjhjeqsddffx45xx1211}
 that $u$ is a bounded solution of \dref{hhhssdddjhjeqsddffx45xx1211}.
 Finally, due to the parabolic H\"{o}lder regularity (see e.g. Theorem 1.3 of \cite{Porzio710}), we may derive
   \dref{hnjgghhgbhnggeqx45xx1211} is held.
\end{proof}
With Lemmata \ref{fghfbglemma4563025xxhjklojjkkkgyhuissddff}--\ref{ffgghhfhhghfbglemma4563025xxhjklojjkkkgyhuissddff} in hand,
by means of standard arguments, we can finally verify the claimed statements on decay of solutions
in the case $a = 0.$

{\bf The proof of Theorem \ref{theoremfgb3}}
Suppose on contrary that \dref{ggbbdffff1.1fghyuisdakkklll} is not held.
Then there exist positive constant  $C_1 > 0$ and $(t_j)_{j\in N}\subset (1,\infty)$ such that $t_j\rightarrow\infty$  as $j\rightarrow\infty$ and
\begin{equation}\label{fgvhnjggssderrhhgbhnggeqx45xx1211}
\|u(\cdot, t_j)\|_{L^\infty(\Omega)}\geq C_1~~~\mbox{for all}~~j\geq N.
\end{equation}
On the other hand, invoking Lemma \ref{ffgghhfhhghfbglemma4563025xxhjklojjkkkgyhuissddff}, in light of the Arzel\`{a}-Ascoli theorem we derive  that
\begin{equation}\label{fgvhnjggsssddsderrhsdddhgbhnggeqx45xx1211}
u(\cdot, t)_{t>1}~~\mbox{is relatively compact in}~~ C^0(\bar{\Omega}).
\end{equation}
we conclude that there exist subsequences of $\{t_j\}$,
still denoted in the same way, such that
\begin{equation}\label{fgvsssderrhnjggssderrhhgbhnggeqx45xx1211}
u(\cdot, t_j)\rightarrow u_{\infty}~~~\mbox{in}~~L^\infty(\Omega)~~~\mbox{as}~~j\rightarrow\infty
\end{equation}
with some nonnegative $u_{\infty} \in C^0(\bar{\Omega})$. However, due to  the decay property \dref{fghfbglemma4563025xxhjklojjkkkgyhuissddff}, we derive  that
\begin{equation}\label{fgvssddfffsderrhnjggssderrhhgbhnggeqx45xx1211}
u(\cdot, t)\rightarrow 0~~~\mbox{in}~~L^1(\Omega)~~~\mbox{as}~~t\rightarrow\infty
\end{equation}
Therefore, combining \dref{fgvsssderrhnjggssderrhhgbhnggeqx45xx1211}
and \dref{fgvssddfffsderrhnjggssderrhhgbhnggeqx45xx1211}, we see that necessarily
$$u_{\infty}\equiv0,$$
which contradicts \dref{fgvhnjggssderrhhgbhnggeqx45xx1211} and thereby proves the first claim in \dref{ggbbdffff1.1fghyuisdakkklll}. The claimed stabilization property
of $v$ can be derived along the same lines, relying on an application of \dref{cz2sedfgg.5g5gghh56789hhjui7ssddd8jj90099}, and on
\dref{kmlgbhnggeqx45xx1211}.

{\bf Acknowledgement}:
 The authors are very grateful to the anonymous reviewers for their carefully reading and valuable suggestions
which greatly improved this work.
This work is partially supported by the
Natural Science Foundation of Shandong Province of China (No. ZR2016AQ17), the National Natural
Science Foundation of China (No. 11601215)
 and the Doctor Start-up Funding of Ludong University (No. LA2016006).

\end{document}